\documentclass[english,a4paper,12pt,reqno]{amsart}
\usepackage{amsmath,amsthm,amsfonts,amssymb}

\usepackage[ansinew]{inputenc}
\usepackage{hyperref}
\usepackage{graphicx}
\usepackage{bm}
\usepackage{a4wide}

\newcommand{\leb}{\operatorname{Leb}}
\newcommand{\dist}{\operatorname{dist}}

\newcommand{\supp}{\operatorname{supp}}

\newcommand{\w} {\omega}
\newcommand{\ep}{\epsilon}
\renewcommand{\P}{\theta_\epsilon^\ZZ}
\newcommand{\x}{\times}
\newcommand{\bw}{\bar\w}
\newcommand{\bomega}{\bar\omega}

\def \RR {{\mathbb R}}
\def \ZZ {{\mathbb Z}}
\def \NN {{\mathbb N}}

\def \cR {\mathcal{R}}

\def \cp {\mathcal{P}}
\def \cb {\mathcal{B}}

\newcommand{\dem}{\begin{proof}}
\newcommand{\cqd}{\end{proof}}

\newcommand{\qand}{\quad\text{and}\quad}

\newtheorem*{Theorem*}{Theorem}
\newtheorem{maintheorem}{Theorem}

\newcommand{\cmt}{\begin{maintheorem}}
\newcommand{\fmt}{\end{maintheorem}}

\newtheorem{maincorollary}[maintheorem]{Corollary}

\newcommand{\cmc}{\begin{maincorollary}}
\newcommand{\fmc}{\end{maincorollary}}

\newtheorem{Theorem}{Theorem}[section]
\newcommand{\ct}{\begin{Theorem}}
\newcommand{\ft}{\end{Theorem}}

\newtheorem{Corollary}[Theorem]{Corollary}
\newcommand{\cco}{\begin{Corollary}}
\newcommand{\fco}{\end{Corollary}}

\newtheorem{Proposition}[Theorem]{Proposition}
\newcommand{\cpr}{\begin{Proposition}}
\newcommand{\fpr}{\end{Proposition}}

\newtheorem{Lemma}[Theorem]{Lemma}
\newcommand{\cle}{\begin{Lemma}}
\newcommand{\fle}{\end{Lemma}}

\newtheorem{Sublemma}[Theorem]{Sublemma}
\newcommand{\csle}{\begin{Sublemma}}
\newcommand{\fsle}{\end{Sublemma}}

\theoremstyle{remark}

\newtheorem{Remark}[Theorem]{Remark}
\newcommand{\cre}{\begin{Remark}}
\newcommand{\fre}{\end{Remark}}

\newtheorem{Definition}[Theorem]{Definition}
\newcommand{\cd}{\begin{Definition}}
\newcommand{\fd}{\end{Definition}}

\title{On the liftability of expanding  stationary  measures}

\author{Jos\'e F. Alves}
\address{Jos\'e F. Alves\\ Departamento de Matem\'atica, Faculdade de Ci\^encias da Universidade do Porto\\
Rua do Campo Alegre 687, 4169-007 Porto, Portugal}
\email{jfalves@fc.up.pt} \urladdr{http://www.fc.up.pt/cmup/jfalves}

\author{Carla L. Dias}
\address{Carla L. Dias\\ Departamento de Matem\'atica, Faculdade de Ci\^encias da Universidade do Porto\\
Rua do Campo Alegre 687, 4169-007 Porto, Portugal}
\email{carlald@fc.up.pt}

\author{Helder Vilarinho }
\address{Helder Vilarinho\\ Universidade da Beira Interior, Centro de Matem\'atica e Aplica\c c\~oes (CMAUBI)\\
Rua Marqu\^es d'\'Avila e Bolama, 6200-001 Covilh\~a, Portugal}
\email{helder@ubi.pt}
\urladdr{http://www.mat.ubi.pt/$\sim$helder}

\date{\today}

\thanks{The authors were  supported by CMUP (UID/MAT/00144/2013) and the project PTDC/MAT-CAL/3884/2014 funded by Funda\c{c}\~ao para a Ci\^encia e a Tecnologia (FCT) Portugal with national (MEC) and European structural funds through the program FEDER, under the partnership agreement PT2020. JFA was also  supported by The Leverhulme Trust VP2-2017-004 Visiting Professorship. HV was also  supported by FCT through Centro de Matem\'atica e Aplica\c c\~oes da Universidade da Beira Interior (CMAUBI), project UID/MAT/00212/2019.}



\begin{document}
\begin{abstract}
We consider random perturbations of a topologically transitive local diffeomorphism of a Riemannian manifold.
We show that if an absolutely continuous ergodic stationary measures is expanding (all Lyapunov exponents positive), then there is a random Gibbs-Markov-Young structure which can be used to lift that measure.
We also prove that if the original map admits a finite number of expanding invariant measures then the stationary measures of a sufficiently small stochastic perturbation are expanding.
 \end{abstract}
\maketitle
\tableofcontents

\section{Introduction}

The existence of \emph{inducing schemes} is a powerful tool to understand ergodic properties of a dynamical system. An interesting situation occurs in the presence of a Markov partition with certain regularity properties for the maps induced on the domains of the partition (Gibbs property). In the remarkable works of L.-S. Young~\cite{Y98,Y99} some procedures were established to extract relevant ergodic information from such inducing  schemes, making use of the so-called Young towers. In view of this systematization these induced elements are often referenced as \emph{Gibbs-Markov-Young (GMY) structures}.

 In the deterministic situation,  GMY structures were used to prove the existence of absolutely continuous invariant probability measures (acip), to study decay of correlations and other statistical features or to prove the statistical stability of a system, among others; see for instance \cite{ACF10,ADLP16,AP10,KKM19,MN05,MN08,MN09,Y98,Y99}. On the other hand, in \cite{ADL13} is studied the geometry of a given acip and the liftability problem in the sense of understanding whenever a given acip implies the existence of a GMY structure with an induced invariant measure that projects to the initial acip.

In this work, we are interested in random perturbations of  discrete-time dynamical systems.
This will be done  considering a parameter space $T$ and  a family of maps $(f_t)_{t\in T}$, typically close to an initial map $f$, and a probability measure $\theta$ on the space $T$ that gives a law of choices in each iteration: we replace the original dynamics $f$ by a near one randomly elected from the available random maps $f_t$, leading to an independent  identically distribution of maps along time.
Stochastic stability issues emerge when we consider a family of measures $(\theta_\ep)_{\ep>0}$ supported on an $\ep$-neighborhood of (the parameter associated to) $f$ and compare the invariant measures associated to the perturbed dynamics -- stationary measures -- with the invariant measures of the deterministic dynamics $f$.
This is mainly given by two major perspectives: an average (``annealed") setting,  natural to formalize in terms of a Markov chain, or an almost sure approach, by considering a skew product dynamics and looking for results along generically random orbits.

Random versions of GMY structures have also been a successful tool to understand ergodic properties of randomly perturbed dynamics, either to prove their stochastic stability in~\cite{AV13}, or to estimate decay of correlations through random orbits and other statistical features in~\cite{ABR19,BaBeM02,BaV96,BBD14,LV18}.

In this work, we consider random perturbations of a   local diffeomorphism  on a compact manifold with no  boundary and exhibiting topological transitive behaviour.
In this setting, the main purpose of this paper is to characterize the family of expanding (all Lyapunov exponents positive) absolutely continuous ergodic stationary (\emph{aces}) measures in the sense that come from a random GMY structure, that is, those stationary measures are liftable. In this way, we also characterize the family of stationary measures accumulating (in the weak$^*$ sense) on the convex hull of the expanding ergodic invariant probability measures.


 The existence of GMY structures starting from expanding aces measures is given by Theorem~\ref{main}. This is proved in two moments. First we prove in Proposition~\ref{NUERO on A} that there exists some subset $A\subseteq M$ on which some power of $f$ satisfies some quite strong random expansivity (non-uniformly expanding on random orbits), and later, in Proposition~\ref{random structure} we prove that non-uniformly expanding on random orbits implies the existence of an induced GMY structure.

In Theorem~\ref{liftmeasure} we prove that the unique (due to topological transitivity of the original map) expanding  aces measure is liftable, in the sense that it can be obtained  as a projection of an induced measure in a random GMY structure with integrable return times with respect to the reference  measure. In some sense, we provide information about the geometry of expanding aces measures.
%

We discuss the stability of expanding measures under random perturbations in Theorem~\ref{stabmain}, where we show that if the invariant measures are expanding, then stochastic stability implies that the same holds for the absolutely continuous ergodic stationary  aces measures. To a better understanding of this question we notice that Lyapunov exponents are typically only defined almost everywhere, with respect to an invariant measure  for determinist dynamical systems, or with respect to a stationary measure   for random dynamical systems; see \cite{A98,O68}. Though assuming a continuous dependence of those measures with respect to the dynamics, it is not true in general that the Lyapunov exponents depend continuously on the map, as the simple case of the one-dimensional quadratic family illustrate; see Subsection~\ref{se.quadratic} for  details.  In this case, we have a sequence of maps with negative Lyapunov exponent whose invariant measures accumulates on an invariant measure for the limiting dynamics with positive Lyapunov exponent.  The situation is different for local diffeomorphisms, where we show that expanding invariant measures are necessarily accumulated by expanding (invariant or stationary) measures.

This paper is organized as follows. In Section~\ref{s:Setting and main results} we introduce basic definitions related to random perturbations and give the main results. In Section~\ref{s:Expanding implies NUERO} we obtain non-uniform expansion from the expansivity hipothesis. This reduces Theorem~\ref{main} to obtaining a GMY structure from this non-uniform expansion (Proposition~\ref{random structure}) which is proved along Section~\ref{s:NUERO _GMY}. In Section~\ref{s:lift} we construct a stationary measure starting from a random GMY structure and answer the liftability question proving Theorem~\ref{liftmeasure}. Theorem~\ref{stabmain} on the stability of expanding measures under random perturbations is proved in Section~\ref{s:stability}.

\section{Setting and main results}\label{s:Setting and main results}

\subsection{Random perturbations}\label{s:Random perturbations}
Let $M$ be a compact Riemannian manifold endowed with a normalized volume measure $\leb$ that we call
Lebesgue measure, and let $f\colon M\to M$ be a $C^2$ local diffeomorphism. The idea of random perturbations is to replace the original deterministic
dynamics $f$ by a dynamic given through maps
$f_t$ elected
independently and identically distributed in each iterate. For this, we consider a metric (parameter) space $T$ endowed with some probability measure $\theta$ defined on the Borel $\sigma$-algebra $\mathcal T$ of $T$, and a continuous map
$$
 \begin{array}{rccl}
 \Phi
:& T &\longrightarrow&  C^2(M,M)\\
 & t &\longmapsto &  f_t
\mathbf{} \end{array}
 $$
such that $f=f_{t^*}$ for some  $t^*\in T$. Consider the product measurable space $
 T^{\mathbb N}$, endowed with the Borel product
$\sigma$-algebra $\mathcal F=\mathcal T^{\mathbb N}$ and the probability product measure
$\theta^{\mathbb N}$. Let  $\sigma: T^{\NN}\to T^{\NN}$ given by $\sigma(\w_1,\w_2,\ldots)=(\w_2,\w_3,\ldots)$ be the shift map and set $\pi\colon T^\NN\to T$ to be the projection on the first coordinate $\pi(\w)=\w_1$. We use the measurable map $\Phi\circ\pi\colon T^\NN\to C^2(M,M)$ to set $f_\omega=(\Phi\circ\pi)(\w)=f_{\w_1}$, for each
$\omega=(\omega_1,\omega_2,\ldots)\in T^{\NN}$.  We define $f_{\omega}^0={\text{Id}}_M$, and for $n\geq 1$,
\begin{equation*}
f_{\omega}^n(x)=
 (f_{\sigma^{n-1}(\omega)}\circ\cdots\circ f_\omega)(x)=(f_{\w_{n-1}}\circ\cdots\circ f_{\w_1})(x).
\end{equation*}
Set $\ZZ_+=\{0,1,2,\ldots\}$. Given $x\in M$ and a \emph{realization} $\omega\in  T^{\NN}$ we call the sequence
$\big(f_{\omega}^n(x)\big)_{n\in{\mathbb Z}_+}$ a {\em random orbit} of $x$. Notice that the realization
$\omega^*=(\tau^*,\tau^*,\ldots)\in T^{\NN}$ gives rise to the unperturbed deterministic orbits given by the
original dynamics~$f$.
We refer to  $(\Phi,\theta)$ as a \emph{random perturbation} of $f$.
A probability measure $\mu$  on the Borel sets of $M$ is called
{\em stationary} (for  $(\Phi,\theta)$) if
$$\iint(\varphi\circ f_t)(x)
\,d\mu(x)d\theta(t)=\int\varphi(x)\,d\mu(x),$$
for all $\varphi:M\to\RR$ continuous.
 We say that a  set $A\subset M$ is {\em random invariant}  if for
$\mu$-almost every (a.e.) $x\in M$ we have
 \begin{enumerate}
\item[]  $x\in A\, \Rightarrow \,f_t(x)\in A$ for  $\theta$-a.e.  $t\in T$; and
\item[] $x\in M\setminus A \Rightarrow f_t(x)\in M\setminus A$ for $\theta$-a.e.  $t\in T$.
 \end{enumerate}
A stationary measure $\mu$ is said to be {\em ergodic} (w.r.t. $(\Phi,\theta)$) if for
every random invariant set $A$ we have
 $\mu(A)=0$ or $\mu(A)=1$.
Denote by $\cb$ the Borel $\sigma$-algebra on $M$, and endow $ T^{\NN}\x M$ with the product $\sigma$-algebra $\mathcal F\otimes\cb$. We define the {\em skew-product} map as the locally constant measurable map
 \begin{equation*}
\begin{array}{rccc} S: &  T^{\NN}\times M &\longrightarrow &
 T^{\NN}\times M\\
 & (\omega, x) &\longmapsto & \left(\sigma(\omega),f_{\omega}(x)\right).
\end{array}
 \end{equation*}
It is well known that
a Borel probability measure   $\mu$ on $M$  is a stationary  measure if and only if $\theta^\NN\times\mu$ is $S$-invariant, and
$\mu$ is ergodic if and only if $\theta^\NN\times\mu$ is $S$-ergodic; see e.g.~\cite{O83}.

We recall the definition of Lyapunov exponents, either in the deterministic and in the random situation, as well the notion of expanding measures.
Given $x\in M$ and $v\in T_xM\setminus\{0\}$ we define the \emph{Lyapunov exponent} (w.r.t. $f$)
\begin{equation}\label{lyap exp det}
\lambda(x,v)=\lim_{n\to+\infty}\frac1n\log\|Df^n(x)v\|,
\end{equation}
whenever the limit exist. The classical Oseledet's Multiplicative Ergodic Theorem (MET) ensures that under the integrability condition $\log^+ \|Df(x)\|\in L^1(\mu_f)$, for some $f$-invariant finite measure $\mu_f$, the Lyapunov exponents are well defined for $\mu_f$-a.e.  $x\in M$. Moreover, if $\mu_f$ is ergodic then there at most $\dim(M)$ possible values for the limit in~\eqref{lyap exp det} for $\mu_f$-a.e. $x\in M$. We say that an $f$-invariant probability measure $\mu_f$ is \emph{expanding} if  all Lyapunov exponents   are positive, i.e. for $\mu_f$-a.e. $x\in M$ and all $v\in T_xM\setminus\{0\}$ we have $\lambda(x,v)>0$.
Consider now a random perturbation $(\Phi,\theta)$ of $f$. Given $(\w,x)\in T^{\NN}\times M$ and $v\in T_xM\setminus\{0\}$ we define the \emph{Lyapunov exponent} (with respect to  $(\Phi,\theta)$)
\begin{equation}\label{lyap exp}
\lambda(\w,x,v)=\lim_{n\to+\infty}\frac1n\log\|Df_\w^n(x)v\|,
\end{equation}
whenever the limit exist. The MET for one-sided random perturbations on manifolds (see e.g.  \cite[Theorem 3.2]{LQ95}) asserts that if $\log^+\|Df_\w(x)\|\in L^1(\theta^\NN\times\mu)$ for some stationary measure $\mu$, then the Lyapunov exponents are well defined for $(\theta^\NN\times \mu)$-a.e.  $(\omega,x)\in  T^{\NN}\times M$. Moreover, if $\mu$ is ergodic, then there at most $\dim(M)$ possible real values for the limit in~\eqref{lyap exp} $(\theta^\NN\times \mu)$-a.e.  $(\omega,x)\in  T^{\NN}\times M$. We say that a stationary measure $\mu$ is \emph{expanding} if all Lyapunov exponents  are positive, i.e. for $(\theta^\NN\times\mu)$-a.e.   \( (\w,x) \)  and all \(v\in T_xM\setminus\{0\}\) we have  \(\lambda(\w,x,v)> 0\).

We notice that if we assume that the random maps $f_\w$ are $C^1$ diffeomorphisms (that is, $\Phi(T)$ is contained in the set of $C^1$ diffeomorfisms in $M$) then the sum of all Lyapunov exponents (counting with multiplicity) is less or equal than zero for $(\theta^\NN\times\mu)$-a.e.  $(\w,x)$, being that the equality holds if $f_\w$ preserves the stationary measure $\mu$ for $\theta^\NN$-a.e.  $\w\in T^\NN$; see \cite[Theorem V.1.3]{Ki86} for the ergodic situation and \cite[Theorem 3.2]{LQ95} for the general case.

\subsection{Random GMY structures and   liftability}\label{s:main}
Consider now a family $(\theta_\ep)_{\ep>0}$ of Borel probability measures on the metric space $T$, which in some sense give the distribution of random choices at each noise level $\ep>0$.  This allows us to define the family $(\Phi,(\theta_\ep)_{\ep>0})$ of random perturbations. The stability issues emerges when we compare the original dynamics with the randomly perturbed situation as the noise level becomes smaller, i.e. as $\ep$ goes to~0. In order to have some topological control on the random maps  we will assume that the supports $\supp(\theta_\ep)$ of the measures $\theta_\ep$ satisfy
\[
\supp(\theta_\ep)\to\{\tau^*\},\,\text{ as } \,\ep\to0.
\]
 We say that a map $f$ is \emph{stochastically stable} with respect to $(\Phi,(\theta_\ep)_{\ep>0})$ if there is a finite number of  ergodic $f$-invariant probability measures $\mu_1,\ldots,\mu_p$ and the weak$^*$ accumulation points of any family of absolutely continuous ergodic stationary measures $(\mu_\ep)_{\ep>0}$ lie in the convex hull of  $\mu_1,\ldots,\mu_p$, when $\ep\to0$. We stress that accordingly this definition we are always assuming that the stationary measures are absolutely continuous and ergodic, which is not necessarily the case of the $f$-invariant measures $\mu_i$, for which we only require the ergodicity.

\subsubsection{Random GMY}
In the following we introduce the induced Gibbs-Markov-Young (abbreviated GMY) structures for a random perturbation $(\Phi,\theta_\ep)$.

\begin{Definition}\label{def GMY}
We say that  $(\Phi,\theta_\ep)$ induces a \emph{random GMY structure} in a ball $\Delta\subset M$, with measurable inducing time ${\mathcal R}:T^\NN\times \Delta \to \NN$, if there exist $0<\kappa<1$ and $K>0$ such that for $\theta_\ep^\NN$-a.e.  $\omega\in T^\NN$:
\begin{enumerate}
\item there is a countable partition ${\mathcal P}_{\omega}$ of open sets of a $\leb$ full measure subset $\mathcal{D}_\omega$ of $\Delta$;
\item the {\emph{return time function}} ${\mathcal R}_\omega: \mathcal{D}_\omega\to{\mathbb N}$, given by $\cR_\w(x) = {\mathcal R}(\omega, x)$, is constant in each
$U_{\omega}\in{\mathcal P}_{\omega}$,
\item the map $F_{\omega}=f_\omega^{{\mathcal R}_\omega}:\Delta\to\Delta$ verifies:
\begin{enumerate}
\item[(i)] ${F}_{\omega}\vert_{U_{\omega}}$ is a $C^2$ diffeomorphism onto $\Delta$, for all
$U_{\omega}\in{\mathcal P}_{\omega}$;
\item[(ii)] for $x$ in the interior of $U_{\omega}\in{\mathcal P}_{\omega}$ we have $\|D
    F_{\omega}(x)^{-1}\| <\kappa;$
\item[(iii)] for every
    $U_{\omega}\in{\mathcal P}_{\omega}$ and $x,y\in U_{\omega}$
    \[
    \log\left|\frac{\det D  F_{\omega}(x)}{\det D  F_{\omega}(y)}\right| \leq
K
    \dist( F_{\omega}(x), F_{\omega}(y)).
    \]
  \end{enumerate}
\end{enumerate}
\end{Definition}

Our first result asserts that expanding aces measures imply the existence of a random GMY structure.


\begin{maintheorem}\label{main} Let \( f: M \to M \) be a \( C^{2} \) local diffeomorphism  such that $f^n$ is topologically transitive for all $n\in\NN$. For each family of expanding  aces measures for random perturbations $(\Phi,(\theta_\ep)_{\ep>0})$ there is some ball $\Delta\subset M$ such that if
$\epsilon>0$ is small enough then $(\Phi,\theta_\ep)$ induces a random GMY structure in $\Delta$.
\end{maintheorem}

We prove Theorem~\ref{main} in two main steps. First we prove that there exists some subset $A\subseteq M$ on which random perturbations of some power of $f$ satisfies some quite strong random expansivity (non-uniformly expanding on random orbits). This is done in Proposition~\ref{NUERO on A}. Secondly, in Proposition~\ref{random structure} we prove that non-uniformly expanding on random orbits implies the existence of an induced GMY structure.

\subsubsection{Liftability} Our result can also be viewed in the context of the so-called ``liftability problem''.
To understand how it works first we give the steps of the construction of a stationary measure from a random GMY structure with integrable return time (for details see Section~\ref{s:lift}):
%
\begin{itemize}
  \item[(i)] Consider two-sided random perturbations by having $\bar\w\in T^\ZZ$.
  \item[(ii)] Having a (two-sided) random GMY structure on a ball $\Delta$ implies the existence of a family $(\nu_{\bar\w})_{\bar\w\in T^\ZZ}$ of finite induced absolutely continuous measures $\nu_{\bar\w}$ on $\Delta$ which define a measure $\nu$ on $T^\ZZ\x\Delta$ that disintegrates as $d\nu(\bar\w,x)=d\nu_{\bar\w}(x)d\theta_\ep^\ZZ(\bar\w)$.
    \item[(iii)] Define the absolutely continuous projection measures on $M$:
\begin{equation*}
{\mu}_{{\bar\w}}=\sum_{j=0}^{+\infty}
(f_{{\sigma}^{-j}({\bar\w})}^j)_*(\nu_{\sigma^{-j}({\bar\w})}\vert\{\bar{\cR}_{{\sigma}^{-j}.
({\bar\w})}>j\}).\end{equation*}
  \item[(iv)] Consider the  measure $\tilde\mu_\ep=\int\bar\mu_{\bar{\w}} \,d\theta_\ep^\ZZ(\bar{\w})$, which is finite due to the $\nu$-integrability of the inducing time function $\bar{\cR}$.
  \item[(v)] Finally, take the normalization  $\mu_\ep=\tilde\mu_\ep/\tilde\mu_\ep(M)$, which is an aces probability measure.
\end{itemize}
In this case, we say that the measure $\nu$ on $T^\ZZ\x\Delta$ that disintegrates on the family $(\nu_{\bar\w})_{\bar{\w}\in T^\ZZ}$ is the lift of the stationary measure $\mu_\ep$ (and $\mu_\ep$ is liftable) and that $\mu_\ep$ is the projection of~$\nu$. It is known that the transitivity of the original map implies the uniqueness of the aces probability  measure. The next result ensures that if this stationary measure is expanding then it is liftable. It is proved in Section~\ref{s:lift}.
\begin{maintheorem}\label{liftmeasure} Let \( f: M \to M \) be a \( C^{2} \) local diffeomorphism such that $f^n$ is topologically transitive for all $n\in\NN$ and consider a family $(\mu_\ep)_{\ep>0}$ of expanding aces measures for $(\Phi,(\theta_\ep)_{\ep>0})$. If $\epsilon>0$ is small enough, then
the unique aces probability measure $\mu_\ep$ is liftable.
\end{maintheorem}




\subsection{Stability of expanding measures}\label{s:intro stability}

We will see that expanding invariant measures can only be accumulated by absolutely continuous ergodic stationary measures if they are expanding. In particular expanding stationary measures emerge under stochastic stability if the $f$-invariant ergodic probability measures are expanding. We stress that the next result holds for the deterministic map $f$ and the random maps $f_t$ in $C^1$.
\begin{maintheorem}\label{stabmain}
Let \( f: M \to M \) be a \( C^{2} \) local diffeomorphism admitting a finite number of expanding ergodic invariant  probability measures $\mu_1,\ldots,\mu_p$. If $f$ is {stochastically stable} with respect to $(\Phi,(\theta_\ep)_{\ep>0})$ then, for small enough $\ep>0$, any aces measure $\mu_\ep$ is expanding.
\end{maintheorem}

From the previous results we can straightforward derive the following.

\begin{maincorollary}
Let \( f: M \to M \) be a \( C^{2} \) local diffeomorphism such that $f^n$ is topologically transitive for all $n\in\NN$ admitting an unique
 expanding ergodic absolutely continuous invariant   probability measure $\mu_0$ and consider a family $(\Phi,(\theta_\ep)_{\ep>0})$ of random perturbations and a family of aces measures $(\mu_\ep)_{\ep>0}$. If $\mu_\ep\to\mu_0$ in the weak$^*$ topology as $\ep\to0$ (stochastic stability), then for $\epsilon>0$ small enough $\mu_\ep$ is liftable.
\end{maincorollary}
\subsubsection{Statistical stability}

Observe that we do not have strong  hypotheses on the measure $\theta_\ep$ that selects the random maps. In particular, those measures can be supported on a single map $g^\ep$ in the $\ep$-$C^1$-neighborhood of $f$, that is,  $\theta_\ep$ is the Dirac measure $\delta_{g^\ep}$ supported on $g^\ep$.
The map $f$ is \emph{statistically stable} if it is stochastically stable with respect to  $(\Phi,(\delta_{g^\ep})_{\ep>0})$.

We recall that an $f$-invariant probability measure $\mu_f$ is an \emph{Sinai-Ruelle-Bowen (SRB)}
measure for $f$ if, for a positive Lebesgue measure set of points $x \in M$,
\begin{equation*}
\lim_{n\rightarrow +\infty}\frac{1}{n}\sum_{j=0}^{n-1}\varphi(f^j(x)) = \int \varphi d\mu,
\end{equation*}
for all continuous $\varphi:M\rightarrow\mathbb{R}$.

Thus if $f$ admits an unique SRB measure $\mu_f$ and the same holds for all maps $g$ in the $\ep$-$C^1$-neighborhood of $f$, admitting a unique SRB measure $\mu_g$, then $f$ is statistically stable if
$
g\mapsto \mu_g
$
is continuous at $f$ in the weak$^*$ topology. From Theorem~\ref{stabmain} we get the following result.

\begin{maincorollary}
If $\mathcal F\subset C^2(M,M)$ is a family of statistically stable local diffeomorphisms, then the subset of maps in $\mathcal F$ having all SRB measures expanding is open in~$\mathcal F$.
\end{maincorollary}

\subsubsection{Maps with critical sets}\label{se.quadratic}
The situation described above  is completely different when we consider maps which are no longer local diffeomorphisms, like maps with critical points.
This comprises the well-known family of one dimensional quadratic maps, where we can easily find a counterexample for our results if we do not assume the local invertibility.
Actually, let $f_a:[0,1]\to[0,1]$ be the family of quadratic maps, defined for $x\in [-1,1]$ and $a\in[0,2]$ as $$f_a(x)=1-ax^2.$$ It was proved by Benedicks and Carleson in \cite{BC85, BC91} that there is a set of positive Lebesgue measure $\mathcal A\subset [0,2]$  such that if $a\in\mathcal A$ then $f_a$ admits an (unique) absolutely continuous  invariant probability measure $\mu_a$, which is the SRB measure for $f_a$. Moreover, it follows from \cite[Theorem 3]{K90a} that $\mu_a$ is an expanding measure for each $a\in\mathcal A$.

On the other hand,
it was proved in \cite[Theorem A]{T01} that for each $a\in\mathcal A$ there is a sequence  $(a_n)_n$ of parameters in $[0,2]$ with $a_n\to a$  as $n\to\infty$ such  $f_{a_n}$ has an attracting periodic orbit for each $n$. This in particular implies that for each $a_n$ the SRB measure $\mu_{a_n}$ of $f_{a_n}$ is a singular measure supported on that attracting periodic orbit, thus having a negative Lyapunov exponent.

Fix now $a^*\in \mathcal A$ and a sequence $(a_n)_n$ as above converging to $a^*$. Define
 $$
 \mathcal F =\big\{ f_a : a\in\{a_n: n\in\NN\}\cup \{a^*\} \big\}.
 $$
Observing that $a^*$ is the only non isolated point and $\mu_{a_n}\to\mu_{a^*}$ in the weak* topology as $n\to\infty$, it easily  follows  that the family $\mathcal F$ is statistically stable. However,  $\mu_{a^*}$ is an ergodic expanding measure, but the measures $\mu_{a_n}$ all have negative Lyapunov exponent, for they are supported on attracting periodic orbits.

\section{Non-uniform expansion along random orbits}\label{s:Expanding implies NUERO}

In this section we prove that having a family of expanding aces measures implies some non-uniform expansion along random orbits generated by some power iterates of a random perturbation (Proposition~\ref{NUERO on A}). First we discuss an ergodic decomposition theorem and give a formalism to handle with power iterates of a random perturbation.

\subsection{Ergodic decomposition revisited}


We recall an ergodic decomposition theorem following \cite{Ki86}. Fix a random perturbation $(\Phi,\theta)$. Denote by $\mathbb{S}$ the set of all stationary measures in $M$ and by $\mathbb{S}_e$ the set of all ergodic stationary measures (w.r.t. $(\Phi,\theta)$). We endow $\mathbb S$ with a measurable structure $\mathcal S$ by considering a map $G(\eta)=\int g\,d\eta$ measurable if $g:M\to\RR_0^+$ is measurable with respect to  the completions of the Borel $\sigma$-algebra for any stationary measure.
\begin{Theorem}\cite[Theorem A.1.1.]{Ki86}
  The set $\mathbb S_e$ is a measurable subset of $\mathbb S$, and to each measure $\eta\in\mathbb S$ corresponds a unique probability measure $\nu_\eta$ on the measurable space $(\mathbb S,\mathcal S)$ such that $\eta$ can be uniquely represented as an integral
  \[
  \eta=\int_{\mathbb S_e}\rho\,d\nu_\eta.
  \]
\end{Theorem}
  The formula above means that for any Borel $B\subset M$ we have
  \[
  \eta(B)=\int_{\mathbb S_e}\rho(B)\,d\nu_\eta(\rho),
  \]
  where $\nu_\eta$ is concentrated on $\mathbb S_e$, or equivalently
  \[
  \int_M \varphi(x)\,d\eta(x)=\int_{\mathbb S_e}\int_M\varphi(x)d\rho(x)d\nu_\eta(\rho)
  \]
  for any bounded measurable function $\varphi\colon M\to\RR$.
Denote by $\mathcal P_e$ the set of $S$-ergodic probability measures on $ T^{\NN}\times M$. Note that $\{\theta^\NN\x\rho\colon\rho\in\mathbb{S}_e\}\subset\mathcal P_e$. From the uniqueness of the ergodic decomposition we may establish a particular decomposition for product measures $\theta^\NN\x\eta$, where $\eta$ is a stationary measure; see  \cite[Remark A.1.2.]{Ki86}.

\begin{Corollary}\label{Remark A.1.2. KI86}
  If $\eta$ has ergodic decomposition $\eta=\int_{\mathbb S_e}\rho\,d\nu_\eta$ then
  $\theta^\NN\times\eta$ has the ergodic decomposition
  \[
  \theta^\NN\x\eta=\int_{\mathcal P_e}\hat\rho\,d\hat\nu_{\theta^\NN\x\eta}(\hat\rho)=\int_{\mathbb S_e}\theta^\NN\times\rho\,d\nu_{\eta}(\rho).
  \]

\end{Corollary}

For our purposes, the case where the number of ergodic stationary measures is finite is of special interest. Assume that $(\Phi,\theta)$ admits only a finite number of ergodic stationary measures $\mathbb S_e=\{\mu_1,\ldots\,\mu_\ell\}$. Then each stationary measure $\eta$ belongs to the convex hull of $\mathbb S_e$ in the space of probability measures on $M$:
\[
\eta=\alpha_1^\eta\mu_1+\cdots+\alpha_\ell^\eta\mu_\ell,
\]
with $\alpha_i^\eta\ge0$, $i=1,\ldots,\ell$, and $\alpha_1^\eta+\cdots+\alpha_\ell^\eta=1$. Setting $$\nu_\eta=\alpha_1^\eta\delta_{\mu_1}+\cdots+\alpha_\ell^\eta\delta_{\mu_\ell},$$
we have
  \begin{align*}
  \theta^\NN\times\eta&=\alpha_1^\eta(\theta^\NN\x\mu_1)+\cdots+\alpha_\ell^\eta(\theta^\NN\x\mu_\ell)\\
  &=\int_{\mathbb S_e}\theta^\NN\x\rho\, d\nu_\eta(\rho),
  \end{align*}
which must be the ergodic decomposition of $\theta^\NN\x\eta$ due to the uniqueness of the measure $\nu_\eta$. Thus, any \emph{ergodic component} of $\theta^\NN\x\eta$ should be a product measure $\theta^\NN\x\rho$, for some ergodic stationary measure $\rho$.

\subsection{Power random perturbations}\label{power}
Given a random perturbation $(\Phi,\theta)$ we introduce the random perturbation  $(\Phi_N,\theta^N)$, for some $N\in\NN$ as follows. We consider the product measurable spaces
$(T^N,\mathcal T^N,\theta^N)$ and $((T^N)^\NN,(\mathcal T^N)^{\mathbb N},(\theta^N)^\NN)=(T_N^\NN,\mathcal F_N,\theta_N^\NN)$,
and the left shift $\tilde\sigma\colon T^{\NN}_N\to T^{\NN}_N$. There is a natural relation between $ T^{\NN}_N$ and $ T^{\NN}$, by identifying $$\w=(\w_1,\w_2,\ldots,\w_{N},\w_{N+1},\ldots)\in T^{\NN}$$ with $$\tilde\w=(\tilde\w_1,\tilde\w_2,\ldots)\in T_N^\NN,\quad\text{ where }\tilde\w_i=(\w_{1+N(i-1)},\w_1,\ldots,\w_{Ni}),$$ for all $i=1,2,\ldots$. In this case we write $\pi_N(\tilde{\w})=\w$. We endow $T^N$ with the product \emph{sup} metric and consider the continuous map
$$
 \Phi_N
: T^N \longrightarrow  C^2(M,M)$$
such that, for $\tau=(t_1,\ldots,t_{N})\in T^N$ we define
$$\Phi_N(\tau)= f_{\tau}=\Phi(t_{N})\circ\cdots\circ \Phi(t_1)=f_{t_{N}}\circ\cdots\circ f_{t_1}.
 $$
We are now interested in the random orbits $(f_{\tilde\omega}^n(x))_{n\ge0}$ of the random perturbation $(\Phi_N,\theta^N)$ of $f^N$, where $f_{\tilde\w}=f_{\tilde\w_1}$ for all $\tilde\w=(\tilde\w_1,\tilde\w_2,\ldots)\in T_N^\NN$. With the previous identification, the random orbits for $(\Phi_N,\theta^N)$ are generated by the $N$-powers of the dynamics of $(\Phi,\theta)$, in the sense that $f_{\tilde\w}^n(x)=f_{\w}^{Nn}(x)$. In particular, $f_{\tilde\w}=f_\w^N$.

A Borel probability measure $\mu_{N}$ will be called a \emph{$N$-stationary measure} if it is a stationary measure with respect to the random perturbation $(\Phi_N,\theta^N)$. An $N$-stationary measure $\mu_{N}$ is \emph{$N$-ergodic} if for every $N$-random invariant set $A$ (meaning random invariant with respect to  $(\Phi_N,\theta^N)$) we have $\mu_{N}(A)=0$ or $\mu_{N}(A)=1$. For the skew-product $S_N: T^{\NN}_N\x M\to  T^{\NN}_N\x M$ we have the identification $S_N(\tilde\w,x)=(\tilde\sigma(\tilde\w),f_{\tilde\w}(x))= (\sigma^N(\w),f_{\w}^{N}(x))=S^N(\w,x)$. In view of this,  we have that   $\mu_{N}$ is a $N$-stationary probability measure if and only if $\theta_N^\NN\times\mu_{N}$ is $S^N$-invariant, and we also have that an $N$-stationary measure $\mu_{N}$ is \emph{$N$-ergodic}  if and only if $\theta_N^\NN\x\mu_{N}$ is $S^N$-ergodic.

It is straightforward to check that if $\mu$ is a stationary measure then it is also an $N$-stationary measure, but the converse is not true in general.
We notice moreover that the ergodic stationary measures are not necessarily $N$-ergodic measures. However, we can have a suitable decomposition of an ergodic stationary measure into $N$-ergodic components.

\begin{Proposition}\label{N decomposition}
  Let $\mu$ be an ergodic stationary measure w.r.t. $(\Phi,\theta)$.  Given $N\ge 1$, there are $k\in\{1,\ldots,N\}$ and $N$-random invariant sets $C_1,\dots,C_k$ such that:
  \begin{enumerate}
  \item $\{C_1,\dots,C_k\}$ is a partition ($\mu$-mod 0) of $M$.
  \item $\mu(C_j)\ge 1/N$ for each $1\le j\le k$.
    \item $\mu\vert C_j$ is $N$-ergodic for each $1\le j\le k$.
  \end{enumerate}
\end{Proposition}

We call $(C_j,\mu\vert C_j)$ the \emph{$N$-ergodic components} of $\mu$.
\begin{proof}
Let $N\ge1$ be given. Applying \cite[Lemma 2.5]{ADL13} to $\theta^\NN\x\mu$ and $S$ we get
$S^N$-invariant Borel sets $A_1,\dots,A_k\subset T^{\NN}\x M$ such that:
  \begin{enumerate}
  \item $\{A_1,\dots,A_k\}$ is a $\theta^\NN\x\mu$ mod 0 partition  of $ T^{\NN}\x M$.
  \item $(\theta^\NN\x\mu)(A_j)\ge 1/N$ for each $1\le j\le k$.
  \item $(S^N,(\theta^\NN\x\mu)\vert A_j)$ is ergodic for each $1\le j\le k$.
  \end{enumerate}
Since for $(\theta^\NN\x\mu)$-mod 0 we have $A_j=\supp((\theta^\NN\x\mu)\vert A_j)$, $j=1,\ldots,k$, we may consider $A_j=\supp((\theta^\NN\x\mu)\vert A_j)$. Since measures $(\theta^\NN\x\mu)\vert A_j$ provide an ($S^N$-) ergodic decomposition for $\theta^\NN\times\mu$, from Corollary~\ref{Remark A.1.2. KI86} we get that this measures $(\theta^\NN\x\mu)\vert A_j$ should be indeed product measures $\theta^    \NN\x\rho_j$, for some $N$-ergodic stationary measures $\rho_j$. We consider then each $C_j$ as $\supp(\rho_j)$, for all $1\le j\le k$.
This implies that
\[\supp((\theta^\NN\x\mu)\vert A_j)=\supp(\theta^\NN\x\rho_j)=\supp(\theta^\NN)\x C_j,\]
and thus
\[
(\theta^\NN\x\mu)\vert A_j = \theta^\NN\x(\mu\vert C_j),
\]
for all $1\le j\le k$. Hence
$\theta^\NN\x(\mu\vert C_j)$
is $S^N$-ergodic and so $(\mu\vert C_j)$ is $N$-ergodic for each $1\le j\le k$. Clearly, $\{C_1,\dots,C_k\}$ is a $\mu$ mod 0 partition  of $M$ and $\mu(C_j)\ge 1/N$ for each $1\le j\le k$.
\end{proof}

\subsection{Expanding stationary measures imply NUERO}\label{exp to NUERO to GMY}
  We say that a family of random perturbations  $(\Phi,(\theta_\ep)_{\ep>0})$  is
\emph{non-uniformly expanding on random orbits} (NUERO) on a set $A\subset M$ if, at
least for sufficiently  small $\ep>0$,  there is ${a_0}>0$ such that for
$\theta_\ep^\NN\times (\leb\vert A)$-a.e.  $(\omega,x)\in T^\NN\times M$
 \begin{equation} \label{NUEOA}
\liminf_{n\to +\infty}\frac{1}{n}
\sum_{j=0}^{n-1}\log\|Df_{\sigma^j(\w)}(f_{\omega}^j(x))^{-1}\|< -{a_0}.
 \end{equation}

The following result is the random counerpart of \cite[Lemma 2.3]{ADL13}.

\begin{Lemma}\label{lem:neg}
 Let \( \mu_\ep \)
    be an expanding aces measure.
There is $c>0$ such that for all sufficiently large \( N \)
\begin{equation}\label{neg}
\int \log \|(Df_\w^{N})^{-1}\| \, d(\theta_\ep^\NN\x\mu_\ep) <-c < 0.
\end{equation}
\end{Lemma}

\begin{proof}
Set $\psi_N(\w,x)=\log\|Df_\w^N(x)\|$, and write $\psi=\psi_1$. Recall that the skew product $S$ is $(\theta_\ep\times\mu_\ep)$ invariant and notice that $\psi_{n+m}=\psi_n+\psi_m\circ S^n$. By the continuity of $\Phi$, the
subadditive ergodic theorem, and the positivity of all Lyapunov exponents,
there exists  \( \lambda>0\) such that for \( (\theta_\ep^\NN\x\mu_\ep) \)-a.e. \( (\w,x) \) we have
 \begin{equation}\label{eq.tag} 
  \lim_{n\to\infty} \frac 1n\log \|Df_\w^{n}(x)^{-1}\|=-\lambda.
  \end{equation}
  In fact this \(  \lambda  \) may be chosen precisely as the smallest Lyapunov exponent.
We define the sequence of sets
\[
B_{N}=\{(\w,x): \log \|Df_\w^{N}(x)^{-1}\|>{-\lambda N}/{2}\}.
 \]
We have that \( (\theta_\ep^\NN\times\mu_\ep)(B_{N})\to 0 \) as \( N\to \infty \) and since \((\theta_\ep^\NN\x\mu_\ep)\) is a probability measure
we have
\begin{equation}\label{limited0}
    \int_{M\setminus B_N}\log\|Df_\w^N(x)^{-1}\|d(\theta_\ep^\NN\x\mu_\ep) \le -\frac\lambda2
N(1-(\theta_\ep^\NN\x\mu_\ep)(B_N)) \leq -\frac \lambda 3 N.\end{equation}

It is therefore sufficient to prove that the integral over $B_N$ is not too large, despite the fact that the integrand is possibly increasing in \( N \). We shall use the following result; see \cite[Sublemma 2.4]{ADL13}.

\begin{Sublemma}\label{critical} Let $\psi \in L^1(\theta_\ep^\NN\x\mu_\ep)$ and let $ (B_n)_n$ be a
sequence of sets with $(\theta_\ep\times\mu_\ep)(B_n)\rightarrow 0$ as $n\to\infty$. Then
$$\frac{1}{n}\sum_{j=0}^{n-1}\int_{B_n}\psi\circ S^j \,d(\theta_\ep^\NN\times\mu_\ep) \to 0, \text{ as $n \to\infty$}.$$
\end{Sublemma}

Returning to the proof of the Lemma,
by the chain rule we have
\begin{equation}\label{limited}
\int_{B_N}\log\|Df_\w^N(x)^{-1}\|d(\theta_\ep^\NN\times\mu_\ep)\leq
\sum_{j=0}^{N-1}\int_{B_N}\log\|Df_{\sigma^j(\w)}(f_\w^j(x))^{-1}\|d(\theta_\ep^\NN\times\mu_\ep)
=: N b_N. 
\end{equation}
By Sublemma~\ref{critical} we get that  $b_N\rightarrow 0$ when $N\to\infty$.
%
%
Therefore we obtain \eqref{neg} provided $N$ is sufficiently large.
\end{proof}

\begin{Proposition} \label{NUERO on A} Let \( f: M \to M \) be a \( C^{2} \) local diffeomorphism and consider a family of random perturbations \((\Phi,(\theta_\ep)_{\ep>0})\) admitting expanding aces measures  $(\mu_\ep)_{\ep>0}$. Then for small enough \(\ep>0\)
  and all $N$ large enough, $(\Phi_N,(\theta_\ep^N)_{\ep>0})$ is NUERO on a $N$-random invariant set $A$ with positive Lebesgue measure.


\end{Proposition}

\dem
From Lemma~\ref{lem:neg} one knows that for all large enough $N$ we have
   \begin{equation*}
     \int\log\|(Df_\w^N(x))^{-1}\|\,d(\theta_\ep^\NN\times\mu_\ep)<-c<0,
    \end{equation*}
for some $c>0$.
From Proposition~\ref{N decomposition}, there must be at least one $N$-ergodic component $(A,\rho_\ep)$ of $\mu_\ep$, where $\rho_\ep=\mu_\ep\vert A$ for some $A\subset M$ with $\mu_\ep(A)\ge1/N$, such that
\begin{equation*}
    \int \log \|(Df_\w^N(x))^{-1}d(\theta_\ep^{\NN}\times\rho_\ep)\|<-c<0,
\end{equation*}
or, equivalentelly,\begin{equation}\label{eq:int Df_w^N -1 drho<0}
     \int\log\|(Df_{\tilde\w}(x))^{-1}\|\,d((\theta_\ep^N)^\NN\times\rho_\ep)<-c<0.
    \end{equation}
    We write $\tilde\psi(\tilde\w,x)=\psi_N(\w,x)$. By Birkhoff's ergodic theorem we have for $((\theta_\ep^N)^\NN\times\rho_\ep)$-a.e. $(\tilde\w,x) \in T_N^{\NN} \times A$
    \begin{align*}
      \lim_{n\to\infty}\frac1n\sum_{j=0}^{n-1} (\tilde\psi\circ (S_{N})^j)(\tilde\w,x) &
            =  \lim_{n\to\infty}\frac1n\sum_{j=0}^{n-1} \log \|(Df_{\tilde\sigma^{j}(\tilde\w)}(f_{\tilde\w}^{j }(x)))^{-1}\|\\
      & = \int\log\|(Df_{\tilde\w}(x))^{-1}\|\,d((\theta_\ep^N)^\NN\times\rho_\ep)\\
      & <-c<0.
    \end{align*}
This means that $(\Phi_N,\theta_\ep^N)$ satisfies the NUERO condition~\eqref{NUEOA} on a set $A \subset M$ with $\mu_\ep(A)>1/N$. Since $\mu_\ep$ is absolutely continuous with respect to $\leb$, we also have $\leb(A)>0$.
\cqd

\section{NUERO implies GMY}\label{s:NUERO _GMY}

Proposition~\ref{NUERO on A} allows us to reduce the proof of the Theorem~\ref{main} to the proof of the following.
\cpr\label{random structure}
 Let $f:M\to M$ be $C^2$ map such that $f$ is topologically transitive  in a random invariant subset $A \subset M$ with positive Lebesgue measure and $(\Phi,(\theta_\ep)_{\ep>0})$ is NUERO in $A$.
 Then there is some ball $\Delta\subset A$ such that if
$\epsilon>0$ is small enough then $(\Phi,\theta_\ep)$ induces a random GMY structure.
\fpr

Indeed, under the assumptions of Theorem~\ref{main}, Proposition~\ref{NUERO on A} gives that for all $N$ large enough, $(\Phi_N,(\theta_\ep^N)_{\ep>0})$ is NUERO on a $N$-random invariant set $A$. Fix such an $N$. By Proposition ~\ref{random structure} we get that  $(\Phi_N,\theta_\ep^N)$ induces a random GMY structure in $\Delta$  with return time $\tilde\cR\colon T_N^\NN\x\Delta\to\ZZ$, provided $\epsilon>0$ is small enough. In particular for $(\theta_\ep^N)^\NN$-a.e. $\tilde\w\in T_N^\NN$ we get a partition $\mathcal{P}_{\tilde\w}$ of $\Delta$ such that the map ${F}_{\tilde\w}=f_{\tilde\w}^{\tilde\cR}$ satisfy the conditions in Definition~\ref{def GMY}. Taking into account the identification of $T^\NN$ and $T_N^\NN$ given in Section~\ref{power}, for $\theta_\ep^\NN$-a.e.  $\w\in T^\NN$ we define a partition $\mathcal{P}_\w$ by taking the same elements of the corresponding $\mathcal{P}_{\tilde\w}$ and setting the return time $\cR\colon T^\NN\x\Delta\to\NN$ as $\cR(\w,x)=N\cdot\tilde\cR(\tilde\w,x)$.


We devote the remaining of this section to the proof of Proposition~\ref{random structure}. Assume that  $f$ is topologically transitive in a random invariant subset $A \subset M$ with positive Lebesgue measure and $(\Phi,(\theta_\ep)_{\ep>0})$ is NUERO in $A$.
First, we introduce the hyperbolic times, main tool for the construction of a random GMY structure and some results about them.

\subsection{Hyperbolic times}
In this subsection  we introduce the notion of hyperbolic time in the random setting and recall some of its main features. For the proofs of the results below see~\cite[Section 2]{AAr03} or \cite[Section 4.1]{AV13}.
\cd
For $0<\lambda<1$, we call $n$ a
{\em $\lambda$-hyperbolic time} for  $(\omega, x)\in T^{\NN}\times M$ if
$$
\prod_{j=n-k+1}^{n}\|Df_{\sigma^j(\omega)}(f^j_{\omega}(x))^{-1}\| \le \lambda^k,
$$
for all $1\le k \le n$.
\fd

 Given $\omega\in T^{\NN}$ and $n\ge 1$, we define
 $$
 H^n_{\omega}=\{x\in M\colon \text{ $n$ is a $\lambda$-hyperbolic time for
 $(\omega, x)$}\}.
 $$
  Notice that if $x \in H_{\w}^j$ for $j\in\NN$, then $f_{\w}^i(x) \in H_{\sigma^i(\w)}^\ell$ for any $1\le i<j$  and
$\ell=j-i$.

\cd\label{de.freq} We say that the {\em frequency of
$\lambda$-hyperbolic times}\index{frequency} for
$(\omega,x)\in T^{\NN}\times M$ is larger than
$\zeta>0$ if, for large $n\in\NN$, there are $\ell \ge\zeta
n$ and integers $1\le n_1<n_2\dots <n_\ell\le n$ which are
$\lambda$-hyperbolic times for $(\omega,x)$, i.e,
$$
\limsup_{n\to +\infty}\frac{1}{n}\# \{1\leq j \leq n: x \in  H^j_{\omega} \}\geq\zeta.
$$
 \fd

\cpr\label{pr.hyperbolic1}
Let $\epsilon >0$ small enough. Assume that $(\Phi,(\theta_\ep)_{\ep>0})$ is NUERO  in a random invariant subset $A \subset M$ with positive Lebesgue measure. Then there are $0<\lambda<1$ and $\zeta>0$
(depending only on ${a_0}$ in \eqref{NUEOA} and on the map $f$)
such that, if $\ep>0$ is small enough, for $\theta_\ep^\NN$-a.e.
$\omega\in T^\NN$ and Lebesgue almost every point $x\in A$, the
frequency of $\lambda$-hyperbolic times
for $(\w,x)$ is larger than $\zeta$. \fpr



Next we present some properties of hyperbolic times that will be useful later.

\cle\label{p.contr} 
There exist $\delta_1,C_0 >0$ such that if $n$ is
    $\lambda$-hyperbolic time for $(\omega,x)\in T^\NN\times M$,
    then there is a neighborhood $V^n_\omega$ of $x$ in $M$ such that:
     \begin{enumerate}
     \item $f_{\omega}^n$ maps $V^n_\omega$
    diffeomorphically onto $B(f_{\omega}^n(x),{\delta_1})$;
\item for every $y,z\in V^n_\omega$ and $1\leq k\leq n$
     $$ \dist(f_{\omega}^{n-k}(y),f_{\omega}^{n-k}(z)) \leq \lambda^{k/2}\dist(f_{\omega}^{n}(y),f_{\omega}^{n}(z)).
     $$
      \item for every $y, z\in V^n_\omega$,
 $$ \log\frac{|\det Df_{\omega}^n (y)|}{|\det Df_{\omega}^n (z)|}
 \le C_0\dist(f_{\omega}^{n}(y),f_{\omega}^{n}(z)).
 $$
      \end{enumerate}
\fle

Moreover, for every $y\in V^n_{\omega}(x)$ we have $\|Df^n_{\omega}(y)^{-1}\|\le\lambda^{n/2}$.
We refer for those sets $ V^n_\omega$ as \emph{hyperbolic
pre-balls} and  $ B(f^{n}_{\omega}(x),\delta_1) $
as \emph{hyperbolic balls}.

\cle \label{co.contraction} %
There is $C_2>0$ such that if $n$ is a $\lambda$-hyperbolic time for $(\w,x)$ and $V_\w^n$ is the corresponding hyperbolic pre-ball, then:
\begin{enumerate}
\item  for any
Borel sets $A_1, A_2\subset V_\w^n$
$$\frac{1}{C_2}\frac{\leb(A_1)}{\leb(A_2)}\leq\frac{\leb(f_\w^n(A_1))}{
\leb(f_\w^n(A_2))}\leq{C_2}\frac{\leb(A_1)}{\leb(A_2)}.$$
 \item there is  $\tau_\w^n >0$ such that for any $x\in
H_\w^n$ one has $B(x,{\tau_\w^n}) \subset V_\w^n$. In particular, every \( H_\w^{n} \) is covered by a finite number of hyperbolic pre-balls.
\end{enumerate}
\fle
The previous result is a standard consequence of the last item of Theorem~\ref{p.contr}; see e.g. \cite[Lemma 3.7]{ADL13}.

\subsection{Induced domain}
 This subsection is devoted to the choice of our domain $\Delta$ where the GMY structure will be defined.

 \begin{Proposition} \label{Ball on A} Assume that $(\Phi,\theta_\ep)$ is NUERO on a random invariant set $A$ with positive Lebesgue measure. So, for any small $\delta_1$ there is a ball $B$ of radius $\delta_1/4$ such that $\leb(B\setminus A)=0$.
\end{Proposition}

For a proof of this result see   \cite[Proposition 2.13]{AV13}.

\begin{Lemma} \label{dense} Assume that $f: M \rightarrow M$ is topologically transitive. Then, there are $p\in B$ and $N_0\in\NN$ such that $\cup_{j=0}^{N_0} f^{-j}\{p\}$ is $\delta_1/4$-dense in $A$.
\end{Lemma}
\begin{proof}
Since $f$ is transitive in $M$, so there is a point in $M$ with dense orbit. Take $N_0$ for which $q, f(q), ...., f^{N_0}(q)$ is $\delta_1/4$-dense in $A$
and $f^{N_0}(q) \in B$. The point $p = f^{N_0}(q)$ satisfies the conclusion of the lemma.
\end{proof}

In order to choose our inducing domain we take a certain $\delta_0$ such that $0< \delta_0\ll \delta_1$ in a sense to be determined.  Henceforth we fix the sets
$$\Delta=
B(p,\delta_0),{\quad} \quad\text{and}\quad \Delta^c= M\setminus\Delta.
$$

Next result says that for every random orbit, every ball of sufficiently large size, i.e.
of radius at least $\delta_1$, contains a subset which is mapped diffeomorphically with bounded distortion onto $\Delta$
within a uniformly bounded number of iterations.

 \cle\label{le.returnball}\cite[Lemma 4.14]{AV13} Let $A$ be a random invariant set.
If $\epsilon$ is sufficiently small, there are $D_0$ and $K_0$ such that 
for every $\omega\in\supp(\theta_\ep^\NN)$ and any ball $B$ of radius $\delta_1$ with $\leb(B\setminus A)=0$ there are an open set $V\subset B$ and $0\le m\le N_0$ for which $f^m_\omega$ maps
$V$ diffeomorphically onto $\Delta$ with bounded distortion: for $x, y\in V$ we have
\begin{equation}\label{bd2}
\log\left|\frac{\det Df_{\omega}^m(x)}{\det Df_{\omega}^m(y)}\right|\leq D_0
\dist(f_{\omega}^m(x), f_{\omega}^m(y)),
\end{equation}
and for each $0\leq j\leq m$
and all $x\in f_{\omega}^{j}(V)$ we have
\begin{equation}\label{bd3}
K_{0}^{-1}\leq \|Df_{\sigma^{j}(\omega)}^{m-j}(x)\|, \|(Df_{\sigma^{j}(\omega)}^{m-j}(x))^{-1}\|,  |\det Df_{\sigma^{j}(\omega)}^{m-j}(x)| \leq K_0.
\end{equation}
\fle


\subsection {The auxiliary partition}\label{se.par}
We will now construct the random partition $\mathcal P_{\omega}$ on the reference ball $\Delta=
B(p,\delta_0)$ for $\theta_\ep^\NN$-a.e.  $\omega\in T^{\NN}$. In particular we choose $\delta_0 > 0$ small so that
\begin{equation}\label{estdelta0}
    2\delta_0K_0 ^{N_0}\sigma^{-N_0}< \delta_1K_0 ^{-N_0}.
\end{equation}

 By Lemma~\ref{p.contr} given some $\omega\in T^{\NN}$ and $x \in H^n_{\omega}$ there exists an hyperbolic pre-ball $V^n_{\omega}(x)$
such that  $f_{\omega}^n(V^n_\omega(x)) = B(f_{\omega}^n(x),{\delta_1})$.
From Lemma~\ref{le.returnball} there are a set  $U^{n,m}_{\omega,x} \subset B(f_{\omega}^n(x),{\delta_1})$ and
an integer $0\le m\le N_0$ such that

\begin{equation}\label{D.candidate2}
 f_\omega^{n+m}(U^{n,m}_{\omega,x})=\Delta.
 \end{equation}
As the condition \eqref{D.candidate2} may in principle hold for several values of $m$, for definiteness we shall always assume that $m$ takes the smallest possible value. Observe that  $U^{n,m}_{\omega,x}$ is associated to $x$, by construction, but does not necessarily contain~$x$.
The sets of the type $U^{n,m}_{\omega,x}$, with $x\in H^n_{\omega}\cap\Delta$, are the natural candidates to be in the partition~$\mathcal P_{\omega}$.

In the sequel, we shall frequently omit the symbols $m$, $x$ or $n$ in the notation
 and simply use $U^n_{\omega,x}$, or even $U_\omega$ to denote an element  $U^{n,m}_{\omega,x}$.
Along the process we will introduce inductively sequences of objects
$(\Delta^n_\omega)_n$, $(\mathcal U^n_{\omega})_n$
and $(S_\omega^n)_n$. For each $n$, $\Delta^n_{\omega}$ will be defined as the set of points which  does not belong to any element of the partition constructed up to time $n$,
$\mathcal U^n_{\omega}$ as the union of elements of the partition constructed at step $n$ and $S^n_{\omega}$ as the finite union of hyperbolic pre-balls at step $n$.
A key point in our argument is property \eqref{Hdel_n} below, which says that every point having a hyperbolic time at a given time $n$ will belong to either to an element of the partition or to some \emph{satellite}. All these  and some other auxiliary objects will be defined inductively in the remaining part of this subsection.

\subsubsection*{First step of induction}
Consider some large integer \(R_0\in\mathbb{N} \) to settle the initial step of the construction process. We simply ignore the dynamics before time $R_0$, that will be determined later in Section~ \ref{sub.ebd} (independent of $\omega$).
Recall that there is a radius $\tau_\w^{R_0}>0$  such that each hyperbolic pre-ball $V^{R_0}_{\omega}(x)$  with $x\in H_\omega^{R_0}$ contains a ball of radius $\tau_\w^{R_0}>0$.
Thus, there are finitely many points
$z_{\omega}^1,\ldots,z_{\omega}^{N_{R_0}}\in H^{R_0}_{\omega}$ such that
$$H^{R_0}_{\omega}\cap \Delta\subset S^{R_0}_{\w}  := V^{R_0}_{\omega}(z_{\omega}^1)\cup\dots\cup
V^{R_0}_{\omega}(z_{\omega}^{N_{R_0}}).$$
We take a maximal subset of points $\{x_0,\ldots, x_{k_{R_0}}\}$ such that the corresponding sets of type (\ref{D.candidate2}) are pairwise disjoint and contained in
$\Delta$, and let
$${\mathcal P}_{\omega}^{R_0}=
\{U^{R_0,m_0}_{\omega,x_0}, U^{R_0,m_1}_{\omega, x_1},\ldots,
U^{R_0,m_{k_{R_0}}}_{\omega,x_{k_{R_0}}}\}.
$$
These sets are precisely the elements of the partition ${\mathcal P}_{\omega}$ constructed in the initial $R_0$
step  of the algorithm. Let
\begin{equation*} \Delta^{R_0}_{\omega} = \Delta\setminus \mathcal  \bigcup_ {U \in \mathcal P_{\omega}^{R_0}} U,
\end{equation*}
that is, the points which do not yet belong to any element of the partition.

For each $0\leq i \leq k_{R_0}$, we define the return
time
\begin{equation*}
{\mathcal R}_\omega(x)=R_0+m_i
\end{equation*}
for each $x\in U_{\omega,x_i}^{R_0,m_i}$.

\subsubsection*{General step of induction}
The general step of the construction follows the ideas in the first step with
minor modifications. Given $n>R_0$, we assume that $\mathcal P^j_{\omega}, S^j_{\omega}, \Delta^j_{\omega}, \{{\mathcal R}_{\omega}=j+m\}$ are defined for all $R_0\le j\le n-1$ and $0\leq m \leq N_0$.  As in the initial step, there is a finite set of points
$z_{\omega}^1,\ldots,z_{\omega}^{N_{n}}\in  H_\omega^n\cap\Delta_\omega^{n-1}$ such that
$$ H^n_{\omega}\cap\Delta_\omega^{n-1}\subset S_{\w}^n :=
V^n_{\omega}(z_{\omega}^1)\cup\dots\cup V^n_{\omega}(z_{\omega}^{N_{n}}).$$
 We choose a maximal subset of points   $\{x_0,\ldots,x_{k_n}\} \in \{z_{\omega}^1,\ldots,z_{\omega}^{N_{n}}\} $
  such that the corresponding  sets
   of type \eqref{D.candidate2}  are  pairwise disjoint contained in
$\Delta^{n-1}_{\omega}$. Then we let
\begin{equation*}
   {\mathcal P}_{\omega}^n = \{U^{n,m_0}_{\omega,x_0}, U^{n,m_1}_{\omega, x_1},\ldots,
U^{n,m_{k_{n}}}_{\omega,x_{k_{n}}}\}.
\end{equation*}
 These are elements of the partition $\mathcal P_\w$ constructed in the $n$-step of the algorithm.
 We also define the set of points of $\Delta$ which do not belong to partition elements constructed up to this point:
 \begin{equation*}
\Delta^n_{\omega}= \Delta\setminus \bigcup_{U \in {\mathcal P}_{\omega}^{R_0}, \ldots, {\mathcal P}_{\omega}^n }^{}U.
\end{equation*}

For each $i=0,\dots,k_n$ and $x\in U_{\omega,x_i}^{n,m_i}$, we set the return time
 $${\mathcal R}_{\omega}(x_i)=n+m_i.$$

 Obviously, \begin{equation}\label{Hdel_n} H^n_{\omega}\cap\Delta\subset
{S}^n_{\omega}\cup \bigcup_{U \in {\mathcal P}_{\omega}^{R_0}, \ldots, {\mathcal P}_{\omega}^n }^{}U.
\end{equation}

More specifically we have that $H^n_{\omega}\cap\Delta_\omega^{n-1}\subset S_{\w}^n$, i.e. all points in \( \Delta_\omega^{n-1} \) which have a hyperbolic time at time \( n \) are ``covered'' by the \emph{satellites} \( S_\omega^{n} \) while the points which have a hyperbolic time at time \( n \) but which are already contained in previously constructed partition elements, are trivially ``covered'' by the union of these partition elements. The inclusion \eqref{Hdel_n} will be crucial to prove the integrability of the return times.
This inductive construction allows us to define the family
%
 \begin{equation*}
    \mathcal P_\w = \bigcup_{n\geq R_0} \mathcal P_\w^n
 \end{equation*}

of pairwise disjoint subsets of $\Delta$.
We are going to prove that $\mathcal P_\w$  forms a $\leb$ mod 0 partition of $\Delta$. In fact, the elements of \(  \mathcal P_\w  \) are automatically disjoint by construction and almost every point \(  x  \) has a basis of arbitrarily small neighborhoods which in time grow to large scale and return to \(  \Delta  \) within a finite number of iterates, and each of these returns is a candidate for the creation of an element of \( \mathcal P_\w \) containing \( x \).  The potential problem is  that each time such an opportunity arises, it may be possible that the region \( U^{n,m}_{\omega,x}  \) which returns either does not contain \(  x  \) or cannot be chosen because it overlaps a previously constructed element of \(  \mathcal P_\w  \).  Thus it is theoretically conceivable a priori that \(  \mathcal P_\w  \) may not have full measure in \(  \Delta  \).

First of all, in the next section we are going to see that  the sets of family \(  \mathcal P_\w  \) has the properties of a random GMY structures.

\subsection {Expansion, bounded distortion and uniformity}\label{sub.ebd}
The return time ${\mathcal R}_{\omega}$ for an element $U_{\omega}$
of the partition ${\mathcal P}_{\omega}$ of $\Delta$ is made by a hyperbolic time $n$ where hyperbolic pre-balls $V^n_{\omega}$ are sent onto hyperbolic balls, plus a time $m\leq N_{0}$ such that $f_{\omega}^{n+m}(V^n_{\omega})$ covers $\Delta$ completely. Recalling the expansion property of hyperbolic pre-balls and \eqref{bd3} we have
\begin{eqnarray*}
  \|Df_{\omega}^{n+m}(x)^{-1}\| &\le& \|Df_{\sigma^n(\omega)}^{m}(f_{\omega}^{n}(x))^{-1}\|.
\|Df_{\omega}^{n}(x)^{-1}\| \\
   &\le &K_{0}\lambda^{n/2}\\
   &\le& K_{0}\lambda^{(R_0-N_{0})/2}.
\end{eqnarray*}
If we take $R_0$ sufficiently large, this is smaller than some $\kappa<1$. Moreover, for $K = D_{0}+
C_{0}K_{0 }$, from Lemmas~\ref{p.contr} and~\ref{le.returnball} we have that for any $x, y\in U_{\omega}$
with return time ${\mathcal R}_{\omega}$,
$$\log\left|\frac{\det Df_{\omega}^{{\mathcal R}_{\omega}}(x)}{\det Df_{\omega}^{{\mathcal R}_{\omega}}(y)}\right|\leq K \dist(f_{\omega}^{{\mathcal R}_{\omega}}(x),
f_{\omega}^{{\mathcal R}_{\omega}}(y)).$$
Clearly, $\kappa$ and $K$ could be
taken the same for all $\omega$.
We moreover notice that we may construct  the open sets $\{{\mathcal R}_\omega=\ell\}$ only depending on $\omega_0,\ldots,\omega_{\ell-1}$. In particular, we can regard $\{{\mathcal R}(\omega,x)=\ell\}$ essentially as a union of rectangles $\{\omega:\omega_0,\ldots,\omega_{\ell-1}\, \textrm{fixed}\}\times\{{\mathcal R}_\omega=\ell\}$, which is $\mathcal F\otimes\mathcal B$-measurable.


It is still necessary to verify that the described algorithm actually produces a $\leb$ mod~0 partition $\mathcal P_\w$  of $\Delta$.
For this, notice that by construction $\Delta\supset\Delta_\omega^{R_0}\supset \Delta_\omega^{R_0+1}\supset\cdots$,
so we only have to check that $\leb (\cap_{n}\Delta_\omega^{n})=0.$ This is a consequence of the following  result whose proof follows as in~\cite[Proposition 4.3]{ADL13}, where the estimates depend ultimately in the expansion and bounded distortion constants of the GMY structure that we assume to be uniform in $\w$.
%
%

\cpr\label{d.prop.Sn} There is a constant $L>0$ such that for $\theta_\ep^\NN$-a.e.  $\w\in T^\NN$ \[\sum\limits_{n=R_0}^{\infty}\leb(S_{\w}^n) < L < \infty.\]
\fpr

Indeed, it follows from Proposition \ref{d.prop.Sn} and
Borel-Cantelli lemma that $\leb$-a.e.
$x\in \Delta$ belongs only to finitely many ${S}_{\w}^{n}$'s, and therefore one can find $n$ such that
 $x\notin {S}_{\w}^j$ for $j\ge n$.
 Since  $\leb$-a.e.  $x\in\Delta$ has infinitely many hyperbolic times, it follows from \eqref{Hdel_n}
that  $x\in U$ for some $U\in\mathcal P_{\w}^{R_0}\cup\cdots\cup \mathcal P_{\w}^n$ and therefore we have
 $\leb (\cap_{n}\Delta_\omega^{n})=0$.

\section{Liftability}\label{s:lift}

\subsection{Induced measures}
Consider a random perturbation  $(\Phi,\theta_\ep)$ inducing a random GMY structure in a ball $\Delta\subset M$, with return time $\mathcal R:T^\NN\times \Delta \to \NN$.
We start by considering an extension of the random perturbations to a two-sided noise driving dynamics. Similarly to the one-sided case, we consider now the product space $T^\ZZ$ and the probability product measures $\theta_\ep^\ZZ$. Let $\bar\sigma:T^\ZZ\to T^\ZZ$ be the two-sided left shift $\bar\sigma(\bar\w)=\bar\sigma(\ldots,\w_{-1},\w_0,\w_1,\ldots)=(\ldots,\w_{0},\w_1,\w_2,\ldots)$ and $\pi^+\colon T^\ZZ\to T^\NN$ the natural projection $\pi^+(\bar\w)=\pi^+(\ldots,\w_{-1},\w_0,\w_1,\ldots)=(\w_0,\w_1,\ldots)$.  We also extend the GMY structure to a two-sided version by considering for each ${\bar{\w}}\in T^\ZZ$ the elements $\mathcal P_{{\bar{\w}}}$ to be exactly the same as $\mathcal P_{\pi^+({\bar{\w}})}$ and taking the return time $\bar\cR({\bar{\w}},x)$ as $\cR(\pi^+({\bar{\w}}),x)$. For the random orbits we consider $f_{\bomega}=f_{\pi^+(\bomega)}$. We stress that in this setting the noise driving dynamics $\bar\sigma$ is invertible but this is not required for the random maps $f_{\bomega}$.

We define now a random induced dynamical system as in \cite{AV13}. Let us consider a family $(\Delta_{\bar\w})_{\w\in T^\ZZ}$ of disjoint copies $\Delta_{\bar\omega}$ of
$\Delta$, associated to $\bar\omega\in T^\ZZ$, and their
partitions $\cp_{{\bar{\w}}}$.  For $x\in\Delta_{{\bar{\w}}}$ we define
$F_{\bomega}(x)=f_{\bomega}^{\bar\cR(\bar\w,x)}(x)\in\Delta_{\sigma^{\bar\cR(\bar\w,x)}({\bar{\w}})}$.
Given $A\subset\Delta_{\bar\omega}$ set
$$F^{-1}(A)=\bigsqcup_{n\in\NN}\left\{x\in\Delta_{\bar\sigma^{-n}
(\bomega) } \colon \bar\cR({\bar\sigma^{-n}(\bomega)},x)=n \qand
F_{\bar\sigma^{-n}(\bomega)}(x)\in A\right\}.$$ Moreover,
given a family
$(\nu_{\bar\sigma^{-n}(\bomega)})_{n\in\NN}$ of measures on
$\displaystyle\bigsqcup_{n\in\NN}\Delta_{\bar\sigma^{-n}(\bomega)}$ and a measurable set $A\subset\Delta_{\bar\omega}$ we define
\begin{equation*}\label{pseudo
pf}F_{*}(\nu_{\bar\sigma^{-n}(\bomega)})_{n\in\NN}(A)=\sum_{n\in\NN}
\nu_{\bar\sigma^{-n}(\bomega)}(F^{-1}(A)\cap\Delta_{\bar\sigma^{-n}
(\bomega) } ).
\end{equation*}
 The next result gives the existence of suitable induced measures. The proof can be found in~\cite[Theorem 2.10]{AV13}. Denote by $\leb_\Delta$ the restriction of the Lebesgue measure to $\Delta$.
\begin{Theorem}

  \label{exist.mu.induced} Consider a random perturbation  $(\Phi,\theta_\ep)$ inducing a (two-sided) random GMY structure in a ball $\Delta\subset M$. For $\theta_\ep^\ZZ$-a.e.  $\bomega\in
T^\ZZ$ there is an absolutely continuous
 finite measure $\nu_{\bar{\w}}$ on $\Delta$ such that
$F_{*}(\nu_{\sigma^{-n}(\bar\omega)})_{n\in\NN}=\nu_{\bar\w}$.
Moreover, there is a constant $K_1>0$ such that
$d\nu_{\bar\w}/d\leb_\Delta\leq K_1$ for $\theta_\ep^\ZZ$-a.e.  $\bomega\in
T^\ZZ$.
\end{Theorem}

We notice that the authors in \cite{AV13} claim an uniform lower bound for $d\nu_{\bar\w}/Leb$ which seems not possible to reach. However either in \cite{AV13} or in this work this former estimate on the lower bound is not needed.
We introduce now the following  map
$$
\Theta\colon T^\ZZ\x\Delta\to T^\ZZ\x\Delta$$ given by
\[\Theta(\bw,x)=(\bar\sigma^{\bar\cR(\bw,x)}(\bw),F_{\bar\w}(x))=(\bar\sigma^{\bar\cR(\bw,x)}(\bw),f_{\bar\w}^{\bar\cR(\bar\w,x)}(x)).
\]
Notice that this is not a skew-product application. Let $\nu$ be a measure on $(T^\ZZ\x\Delta,\mathcal T^\ZZ\bigotimes\mathcal{B}_\Delta)$, where $\mathcal{B}_\Delta$ denotes the Borel $\sigma$-algebra of $\Delta$, that disintegrates by $(\nu_{\bw})_{\bw\in T^\ZZ}$. That is, for each measurable set $A\subset T^\ZZ\x\Delta$, we have $$\nu(A)=\int_{T^\ZZ}\nu_{\bw}(A_{\bar\w})\,d\theta_\ep^\ZZ(\bar\w),$$
where $A_{\bw}=\{x\in\Delta\colon (\bw,x)\in A\}$.

\begin{Lemma}
  The measure $\nu$ is finite, $\Theta$-invariant and absolutely continuous with respect to $\theta_\ep^\ZZ\x\leb_\Delta$.
\end{Lemma}

\begin{proof}
The finiteness follows from the fact that there exists $K_1>0$ such that $d\nu_{\bar\w}/d\leb_\Delta<K_1$. Moreover, if $(\theta_\ep^\ZZ\x\leb_\Delta)(C)=0$ for any measurable set $C\subset T^\ZZ\x\Delta$, then for $\theta_\ep^\ZZ$-a.e. $\bar\w\in T^\ZZ$ one should have $\leb_\Delta(C_{\bar\w})=0$ and thus $\nu_\w(C_{\bar\w})=0$, which implies $\nu(C)=0$ and so $\nu\ll(\theta_\ep^\ZZ\x\leb_\Delta)$. For the invariance it is enough to prove that $(\Theta_*\nu) (A\x B) = \nu(A\x B)$ for any measurable set $A\x B\subset T^\ZZ\x\Delta$:
\begin{align*}
  (\Theta_*\nu) (A\x B)
  & = \nu(\Theta^{-1} (A\x B)) \\
   & =\sum_{n=1}^\infty \int_{\sigma^{-n}(A)} ((F_{\bar\w}\vert_{\Delta_{\bar\w}})_*\nu_{\bar\w})(B) d\P({\bar\w})\\
   & =\sum_{n=1}^\infty \int_{A} ((F_{\sigma^{-n}({\bar\w})}\vert_{\Delta_{\sigma^{-n}({\bar\w})}})_*\nu_{\sigma^{-n}({\bar\w})})(B) d\P({\bar\w})\\
   & =\int_A \sum_{n=1}^\infty ((F_{\sigma^{-n}({\bar\w})}\vert_{\Delta_{\sigma^{-n}({\bar\w})}})_*\nu_{\sigma^{-n}({\bar\w})})(B) d\P({\bar\w})\\
   & =\int_A \nu_{\bar\w}(B) d\P({\bar\w})\\
   & =\nu(A\x B).
\end{align*}
\end{proof}

\subsection{Integrability of the inducing times}
\label{s.inducing}

In this section we aim to obtain the integrability of the return time function $\bar{\cR}(\w,x)$ with respect to the measure $\nu$ taking into account the partition we constructed for the GMY structure in~Section~\ref{s:NUERO _GMY}. For a better understanding
we bring back our abstract setting and extend auxiliary sets to the two-sided perturbations.
Set a disk $\Delta$ of radius $\delta_0$
around a point $p$, $S_{\bar\w}^n=S_{\pi^+(\bar\w)}^n$ and $H_{\bar\w}^n=H_{\pi^+(\bar\w)}^n$ sets defined for all $n\geq R_0$, and for convenience we consider $S_{\bar\w}^n=\Delta$ for
$n<R_0$. These sets satisfy the following properties:
\begin{enumerate}
                \item[(a)] if $x \in H_{\bar\w}^j$ for $j\in\NN$, then $f_{\bar\w}^i(x) \in H_{\bar\sigma^(\bar\w)}^\ell$ for any $1\le i<j$  and
$\ell=j-i;$
                \item[(b)]  there is $\zeta>0$
such that for $\leb_\Delta$-a.e.  $x\in \Delta$,
$$\displaystyle\limsup _{n\to\infty}\frac{1}{n}\#\{1\leq j\leq n:
x\in H_{\bar\w}^j\} \geq \zeta;$$
                 \item[(c)]$(H_{\bar\w}^n\cap\{{\bar\cR}_{\bar\omega}\geq n\})\subset  (S_{\bar\w}^n \cup \{\bar{\mathcal R}_{\bomega}= n+m\})$
for some  $0\leq m\leq N_0;$
                \item[(d)] $\displaystyle\sum_{n=R_0}^{\infty}\leb_\Delta(S_{\bar\w}^n)<\infty.$
              \end{enumerate}
Moreover we introduce some notation.
For $(\bar\w,x)\in T^\ZZ\x \Delta$, $x$ may undergo several full returns to \( \Delta \) before time \( n \).  Then we define the following quantities:
\begin{align*}
 H^{(n)}(\bar\w,x) &:= \#\{j\le n\colon j \text{ is an hyperbolic time for \( (\bar\w,x)\)} \}
 \\
  S^{(n)}(\bar\w,x) &:=\#\{j\le n\colon x \text{ belongs to a satellite $S_{\bar\w}^j$}\}
  \\
   R^{(n)}(\bar\w,x)& \text{ to be the number of returns of \( x \) before time \( n \)}
\end{align*}
\begin{Lemma}\label{RSH}  Assume (a)-(d) hold for $\theta_\ep^\ZZ$-a.e. $\bar\w\in T^\ZZ$. There exists some constant \( \eta>0 \)  such that for $\leb_\Delta$-a.e. $x\in\Delta$ and all $n\in\NN$,
\[  \eta R^{(n)}({\bar\w},x)+S^{(n)}({\bar\w},x) \geq  H^{(n)}({\bar\w},x).
 \]
\end{Lemma}
\begin{proof}
 It follows from item (d) that $\leb_\Delta$-a.e. $x\in \Delta$ belongs just to finitely many $S_{{\bar\w}}^{i}\,'{s}$.
Define
$$
s(\bar\w,x)= \#\{i \in \NN : x\in S_{{\bar\w}}^{i}\}.
$$
Let $(\bar\w,x)$ be a $(\theta_\ep^\ZZ\x\leb_\Delta)$-generic point as before. Define $j_0=0$ and, for every $i\in
\mathbb{N}$,
$$j_{i}=j_{i}(\bar\w,x)= j_{i-1}+\bar{\mathcal{R}}(\bar\sigma^{j_{i-1}}(\bar\w),F_{\bar\w}^{j_{i-1}}(x)).$$ This means that
${F}_{{\bar\w}}^{i}(x)= f_{\bar\w}^{j_{i}}(x)$. We define the set of return times for the random orbit $(f_{\bar\w}^n(x))_{n\in\NN}$
$$\mathcal{I}=\mathcal{I}(\bar\w,x)=\{j_{0},j_{1},j_{2},...\}.$$
Let $j\in \NN$ be such that $j_r < j <j_{r+1}$. From item (a), for each $x\in
H_{{\bar\w}}^{j}$ we have ${F}_{{\bar\w}}^{r}(x)=f_{\bar\w}^{j_r}(x)\in H_{\sigma^{j_r}(\bar\w)}^{\ell}$, where $\ell=j-j_{r}$. Assuming that
 $R_0\le \ell<\bar{\mathcal{R}}(\bar\sigma^{j_r}(\bar\w),{F}_{{\bar\w}}^r(x))-N_0=j_{r+1}-j_r-N_0$, then according to our construction, from item (c) we must have ${F}_{{\bar\w}}^r(x)\in S_{\bar\sigma^{j_r}(\bar\w)}^\ell$.
Hence
$$
\#\left\{j\in\{j_{r}+1,...,j_{r+1}-1\}:x \in H_{{\bar\w}}^{j}\right\}\leq
R_0+N_0 + s(\bar\sigma^{j_r}(\bar\w),{F}_{{\bar\w}}^{r}(x)).
$$
Thus, for each $n\in\mathbb{N}$ we may write
\begin{eqnarray*}
\#\{j\leq n:x \in H_{{\bar\w}}^{j}\}&\leq& \#\{j\leq n:j\in\mathcal{I} \} +
\#\{j\leq n:j\in\!\!\!\!\!/ \, \mathcal{I} \}.\nonumber\\
\end{eqnarray*}
That is,

\begin{eqnarray*}
\#\{j\leq n:x \in H_{{\bar\w}}^{j}\}&\leq&  R^{(n)}(\bar\w,x) +
\sum_{k=1}^{ R^{(n)}(\bar\w,x) }[R_0+N_0 + s(\bar\sigma^{j_k}(\bar\w),{F}_{{\bar\w}}^{k}(x))]\nonumber\\
&\leq&  (1+R_0+N_0) R^{(n)}(\bar\w,x)  + \sum_{k=1}^{ R^{(n)}(\bar\w,x) }s(\bar\sigma^{j_k}(\bar\w),{F}_{{\bar\w}}^{k}(x)).
\end{eqnarray*}
Notice  that $$\sum_{k=1}^{R^{(n)}(\bar\w,x)}s(\bar\sigma^{j_k}(\bar\w),F_{{\bar\w}}^{k}(x)) = S^{(n)}({\bar\w},x).$$ Therefore,
\begin{equation}\label{integralretorno}
H^{(n)}({\bar\w},x)\leq
R^{(n)}({\bar\w},x)\left(1+R_0+N_0\right)+
S^{(n)}({\bar\w},x).
\end{equation}
It is enough to take $\eta=1+R_0+N_0$.
\end{proof}

\begin{Proposition}\label{pr.a-d}
$\bar\cR$ is $\nu$-integrable.
\end{Proposition}
\begin{proof}
We have that (a)-(d) hold for $\theta_\ep^\ZZ$-a.e.  $\bar\w\in T^\ZZ$. Since $\nu\ll (\theta_\ep^\ZZ\x\leb_\Delta)$, from  Lemma~\ref{RSH} and we have that for $\nu$-a.e. $(\bar\w,x)$ we have
\[
\frac{\eta R^{(n)}({\bar\w},x)}{n}+\frac{S^{(n)}({\bar\w},x)}{n} \geq \frac{H^{(n)}({\bar\w},x)}{n}.
\]
 Recalling that hyperbolic times have uniformly positive asymptotic frequency, there exists a constant \( \zeta>0 \) such that  \( H^{(n)}(\bar\w,x)/n \geq \zeta \) for all \( n \) sufficiently large, inequality above  gives
 \begin{equation}
 \label{key}
 \frac{R^{(n)}({\bar\w},x)}{n}\left(1+\frac1\eta\frac{S^{(n)}({\bar\w},x)}{R^{(n)}({\bar\w},x)}\right) \geq  \frac\theta\eta > 0.
    \end{equation}
 The ratio  $S^{(n)}({\bar\w},x) / R^{(n)}({\bar\w},x)$ is the average number of times that the points belong to satellites before they return
and, by Birkhoff's ergodic theorem, there is an integrable map $Z\colon T^\ZZ\x\Delta\to\RR$ such that for $\nu$-a.e.  $({\bar\w},x)\in T^\ZZ\x\Delta$,
$$\lim_{n\to\infty} \frac{S^{(n)}({\bar\w},x)}{R^{(n)}({\bar\w},x)}= \lim_{n\to\infty}\frac1n\sum_{j=0}^{n-1} (s\circ\Theta) ({\bar\w},x) = Z({\bar\w},x)$$
and $\int Z d\nu=\int s d\nu<\infty$. On the other hand, \( n/R^{(n)}({\bar\w},x) \) is precisely the average return time over the first \( n \) iterations. By Birkhoff's ergodic theorem there is an integrable map $Y\colon T^\ZZ\x\Delta\to\RR$ sub that for $\nu$-a.e.  $({\bar\w},x)\in T^\ZZ\x\Delta$,
$$\lim_{n\to\infty} \frac{n}{R^{(n)}({\bar\w},x)}= \lim_{n\to\infty}\frac1n\sum_{j=0}^{n-1} (\bar\cR\circ\Theta) ({\bar\w},x) = Y({\bar\w},x),$$
and $\int Y\,d\nu=\int\bar\cR\,d\nu$.
Taking the limit in~\eqref{key} as $n\to\infty$ we get that for $\nu$-a.e.  $({\bar\w},x)$,
\begin{equation*}
  \frac1{Y({\bar\w},x)}\left(1+Z({\bar\w},x)\right) \geq  \frac\theta\eta >0
    \end{equation*}
which implies
\begin{equation*}
 \int Y({\bar\w},x)\,d\nu \leq \frac{\eta}\theta\int 1+ Z({\bar\w},x)\,d\nu <\infty,
    \end{equation*}
so that $\bar\cR$ is $\nu$-integrable.
\end{proof}

\subsection{Stationary measure from GMY}\label{s:stationary from GMY}


Once we have a GMY structure with $\nu$-integrable return time we can construct an absolutely continuous ergodic stationary probability measure $\mu_\epsilon$, which is the unique due to the topological transitivity of the original map $f$.

\cpr\label{GMY implies stationary}
Let $f:M\to M$ be a topologically transitive $C^2$ map. If the random perturbation $(\Phi,\theta_\ep)$ of $f$ induces a random GMY structure in some $\Delta\subset M$ with return time $\bar\cR\in L^1(\nu)$, then $(\Phi,\theta_\ep)$ admits an unique absolutely continuous ergodic stationary probability measure $\mu_\epsilon$.
\fpr

A similar result was given in \cite{AV13} under the assumption of uniform decay of the Lebesgue measures of the sets $\{\bar{\cR}_\w >n\}$.
We are now going to prove Proposition~\ref{GMY implies stationary}. Let $f:M\to M$ be a $C^2$ map and assume that the random perturbation $(\Phi,\theta_\ep)$ of $f$ induces a random GMY structure in some $\Delta\subset M$ with return time $\bar\cR\in L^1(\nu)$. Let us construct the unique absolutely continuous ergodic stationary probability measure $\mu_\epsilon$. We define the family  $(\mu_{\bar\w})_{\bar\w\in T^\ZZ}$ of finite Borel measures on $M$ by
\begin{equation*}\label{mu w nao normalizada}
{\mu}_{{\bar\w}}=\sum_{j=0}^{+\infty}
(f_{{\sigma}^{-j}({\bar\w})}^j)_*(\nu_{\sigma^{-j}({\bar\w})}\vert\{\bar\cR_{{\sigma}^{-j}
({\bar\w})}>j\}),\end{equation*}
where $\bar\cR_{\bar\w}(x)=\bar\cR(\bar{\w},x)$ and the measures $\nu_{\sigma^{-j}({\bar\w})}$ are given by Theorem
\ref{exist.mu.induced}.
The absolute continuity of the measures $(\nu_{{\bar\w}})_{{\bar\w}\in T^\ZZ}$ implies that the measures
of the family $(\mu_{{\bar\w}})_{{\bar\w}\in T^\ZZ}$ are absolutely continuous, and the  property $F_{*}(\nu_{\sigma^{-n}(\bar\omega)})_{n\in\NN}=\nu_{\bar\w}$ implies the quasi-invariance ${f_{\bar\w}}_*\mu_{\bar\w}=\mu_{\sigma({\bar\w})}$.

\cre\label{pastdependence}
By construction, all the measures in the family
$(\nu_{\bar\w})_{\bar\w}$ depend only
in the past
${\bar\w}^-=(\ldots,{\bar\w}_{-2},{\bar\w}_{-1})$ of ${\bar\w}$. Moreover, for ${\bar\w},\bar\tau\in T^\ZZ$ with the same past the
sets $\{\cR_{{\sigma}^{-j}({\bar\w})}=k\}$ and $\{\cR_{{\sigma}^{-j}(\bar\tau)}=k\}$, for
$1\le k \le j$, can be considered exactly the same (as subsets of $\Delta\subset M$).
The
measures ${\mu}_{\bar\w}$ involve sums of the type
$(f_{{\sigma}^{-j}({\bar\w})}^j)_*(\nu_{\sigma^{-j}({\bar\w})}\vert\{\bar\cR_{{\sigma}^{-j
}({\bar\w})}>j\})$,
and since
$$\leb_\Delta(\{\bar\cR_{{\sigma}^{-j}({\bar\w})}>j\})=\leb_\Delta\left(\Delta\setminus\left\{
\displaystyle\bigcup_{k=1}^{j}\{\cR_{{\sigma}^{-j}({\bar\w})}=k\}\right\}
\right)$$ and
$\nu_{\sigma^{-j}({\bar\w})}\ll \leb_\Delta$, the measures $\mu_{\bar\w}$ depend only on the past ${\bar\w}^-$ of ${\bar\w}$.
\fre

\cle Consider the measure $\tilde\mu_\ep=\int\tilde\mu_{\bar\w} \,d\theta_\ep^\ZZ({\bar\w})$. Then $\mu_\ep=\tilde\mu_\ep/\tilde\mu_\ep(M)$ is an   aces probability measure.
\fle
\begin{proof}
Since $(\tilde\mu_{\bar\w})_{{\bar\w}\in T^\ZZ}$ almost surely depend only on the past,
then $\tilde\mu_\ep=\int\tilde\mu_{\bar\w} \,d\theta_\ep^\ZZ({\bar\w})$ satisfies
$\iint(\varphi\circ f_t)(x)
\,d\tilde\mu_\ep(x)d\theta(t)=\int\varphi(x)\,d\tilde\mu_\ep(x),$
for all $\varphi:M\to\RR$ continuous.
Moreover, $\tilde\mu_\ep$ is absolutely continuous due to the absolute
continuity of
 measures $(\tilde\mu_{\bar\w})_{{\bar\w}\in T^\ZZ}$.
For the finiteness of $\tilde\mu_\ep$ we have
\begin{align*}
  \tilde\mu_\ep(M)&=\int\tilde\mu_{\bar\w}(M)\,d\theta_\ep^\ZZ({\bar\w})\\
  &=\int\sum_{
n=0}^{+\infty}\nu_{\sigma^{-n}({\bar\w})}(\{R_{\sigma^{-n}({\bar\w})}>n\})\,d\theta_\ep^\ZZ({\bar\w})\\
&=\int\sum_{
n=0}^{+\infty}\sum_{k=n+1}^{+\infty}\nu_{\sigma^{-n}({\bar\w})}(\{R_{\sigma^{-n}({\bar\w})}=k\})\,d\theta_\ep^\ZZ({\bar\w})\\
&=\sum_{
n=0}^{+\infty}\sum_{k=n+1}^{+\infty}\int\nu_{\bar\w}(\{R_{\bar\w}=k\})\,d\theta_\ep^\ZZ({\bar\w})\\
& =\sum_{
n=0}^{+\infty}\int n\,\nu_{{\bar\w}}(\{R_{{\bar\w}}=n\})\,d\theta_\ep^\ZZ({\bar\w})\\
&=\int \bar\cR({\bar\w},x)d\nu<\infty.
\end{align*}

 We now normalize $\tilde\mu_\ep$ and define an
absolutely continuous stationary measure
$\mu_\ep=\tilde\mu_\ep/\tilde\mu_\ep(M)$.
From \cite[Proposition 2.14]{AV13} we have that
  $\mu_\ep$ is the
unique aces probability measure.
%
%
\end{proof}

We are in conditions to complete the proof of Theorem~\ref{liftmeasure}.
From Theorem~\ref{main} we obtain a GMY structure for a random perturbation $(\Phi,\theta_\ep)$ on a ball $\Delta\subset M$, provided  $\ep>0$ is small enough. Following Section~\ref{s:lift} we consider two-sided perturbations and from Theorem~\ref{exist.mu.induced} there is a finite measure $\nu$ defined on $T^\ZZ\x\Delta$ which is $\Theta$-invariant. From  Proposition~\ref{pr.a-d} the return time $\bar\cR$ is $\nu$-integrable and from Proposition~\ref{GMY implies stationary} we can project $\nu$ to the unique absolutely continuous ergodic stationary probability measure $\mu_\epsilon$.

\section{Stability of expanding measures}\label{s:stability}

We are now going to show that expanding
invariant measures can only be accumulated by expanding stationary measures  as we state in Theorem~\ref{stabmain}.
We identify $T_N^\NN$ and $T^\NN$ as in Section~\ref{power}. Given $(\tilde\w,x)\in T^{\NN}_N\times M$ and $v\in T_xM\setminus\{0\}$, we define the Lyapunov exponent (with respect to  $(\Phi_N,\theta_\ep^N)$)
\begin{equation}\label{N lyap exp}
\lambda_N(\tilde\w,x,v)=\lim_{n\to+\infty}\frac1n\log\|Df_{\tilde\w}^n(x)v\|=\lim_{n\to+\infty}\frac1n\log\|Df_{\w}^{nN}(x)v\|,
\end{equation}
whenever the limit exist. It is straightforward that for all $N\ge1$, $(\w,x)\in T^{\NN}\x M$ and $v\in T_xM\setminus\{0\}$, we have
\begin{equation}\label{lemma:lyap vs N lyap}
\lambda_N(\tilde\w,x,v)=N \lambda(\w,x,v).
\end{equation}
%
We are now in conditions to finish the proof of the Theorem~\ref{stabmain}.

\begin{proof}(of Theorem~\ref{stabmain})
Let $(\mu_\ep)_{\ep>0}$ be a family of stationary ergodic measures whose weak* accumulation points when $\ep\to0$ lie in the convex hull of $f$-invariant ergodic expanding probability measures $\mu_1,\ldots,\mu_p$. From \cite[Lemma 2.3]{ADL13} one knows that, for each $\mu_i$, $i=1,\ldots,p$ exists $N_i$ such that for $n\geq N_i$, we have
    \begin{equation}\label{eq:int log DfN-1 d mu_i<0}
      \int\log\|(Df^n(x))^{-1}\|\,d{\mu_i}<-c_1<0,
    \end{equation}
    for some constant $c_1=c_1(N)>0$. This is the motivation for the introduction of higher iterates.
    Then for $N$ large enough $(N=\max\{N_1, \ldots, N_p\})$ we have
    \begin{equation}\label{eq:int log DfN-1 d mu_i<0}
      \int\log\|(Df^N(x))^{-1}\|\,d{\mu_\infty}<-c_1<0,
    \end{equation}
    where $\mu_\infty$ is the accumulation point of the family $(\mu_\ep)_{\ep>0}$ in the weak$^*$ topology, as $\ep\to0$.
    Fix $N$. There is $\ep_1>0$ such that for all $0<\ep<\ep_1$ we have

    $$
    \int\log\|(Df^N(x))^{-1}\|\,d\mu_\ep<-c_2<0,
    $$
    for some $c_2>0$. Since we are dealing with $C^2$ maps ($C^1$ is enough here) and $N$ is fixed, there is $0<\ep_0<\ep_1$ such that for all $0<\epsilon<\epsilon_0$ and all $\w\in T^{\NN}$ we have for some $c>0$,
   \begin{equation*}
     \int\log\|(Df_\w^N(x))^{-1}\|\,d(\theta_\ep^\NN\times\mu_\ep)<-c<0.
    \end{equation*}
From the proof of  Proposition~\ref{NUERO on A}
 we have that there exists an $N$-ergodic component $(A,\rho_\ep)$ of $\mu_\ep$, where $\rho_\ep=\mu_\ep\vert A$ for some $A\subset M$ with $\mu_\ep(A)\ge1/N$, such that
   for $((\theta_\ep^N)^\NN\times\rho_\ep)$-a.e.  $(\tilde\w,x)$
        \begin{equation*}
        \lim_{n\to\infty}\frac1n\sum_{j=0}^{n-1} \log \|(Df_{\tilde\sigma^{j}(\tilde\w)}(f_{\tilde\w}^{j }(x)))^{-1}\| <0.
      \end{equation*}
                   Since \[\log\|(Df_{\tilde\w}^{n}(x))^{-1}\| \leq  \sum_{j=0}^{n-1} \log \|(Df_{\tilde\sigma^{j}(\tilde\w)}(f_{\tilde\w}^{j}(x)))^{-1}\|, \]
        we have, for $((\theta_\ep^N)^\NN\times\rho_\ep)$-a.e.  $(\tilde\w,x)$ and $v\in T_xM\setminus\{0\}$,
          \begin{align*}
         \lambda_{N}(\tilde\w,x,v)&=\lim_{n\to\infty} \frac1n \log\|Df_{\tilde\w}^{n }(x)v\|\\
        &\geq\lim_{n\to\infty} \frac1n \log\frac{1}{\|(Df_{\tilde\w}^{n }(x))^{-1}\|} \|v\|\\
        &=\lim_{n\to\infty} -\frac1n \log{\|(Df_{\tilde\w}^{n}(x))^{-1}\|}+\frac1n\log \|v\|\\
        &\geq-\lim_{n\to\infty}\frac1n\sum_{j=0}^{n-1} \log \|(Df_{\tilde\sigma^{j}(\tilde\w)}(f_{\tilde\w}^{j}(x)))^{-1}\|\\
        &>0.
        \end{align*}
       Then, from~\eqref{lemma:lyap vs N lyap} we have  for $(\theta_\ep^\NN\times\rho_\ep)$-a.e.  $(\w,x)$ and $v\in T_xM\setminus\{0\}$ that
       \begin{equation}\label{eq:lyap > 0}
         \lambda(\w,x,v)=\frac1N\lambda_{N}(\tilde\w,x,v)>0.
       \end{equation}
       That is, \eqref{eq:lyap > 0} holds for a subset $C\subset T^{\NN}\x M$ of points $(\w,x)$ with $(\theta_\ep^\NN\x\mu_\ep)(C)>1/N$. However, since $(\theta_\ep^\NN\times\mu_\ep)$ is $S$-ergodic the Lyapunov exponents are constant $(\theta_\ep^\NN\times\mu_\ep)$-almost everywhere. Hence for $(\theta_\ep^\NN\x\mu_\ep)$-a.e.  $(\w,x)\in T^\NN\x M$ and all $v\in T_xM\setminus\{0\}$ we have $\lambda(\w,x,v)>0$ and so $\mu_\ep$ is expanding.
\end{proof}

\bibliographystyle{novostyle}

\begin{thebibliography}{BBM02}

%


%
\bibitem{AAr03} J. F. Alves, V. Ara\'ujo, {\em Random
perturbations of non-uniformly expanding maps}, Ast\'erisque {\bf 286} (2003),
25--62.


\bibitem{ABR19}J. F. Alves, W. Bahsoun, M. Ruziboev, \emph{Almost sure rates of mixing for partially hyperbolic attractors}. Preprint 2019.
%

%
\bibitem{ACF10} J. F. Alves, M. Carvalho, J. M. Freitas, {\em
Statistical Stability and Continuity of SRB Entropy for Systems with Gibbs-Markov Structures},
Comm. Math. Phys.  {\bf 296} (2010),
739--767.



\bibitem{ADL13} J. F. Alves, C. L. Dias, S. Luzzatto, {\em Geometry of expanding absolutely continuous invariant measures
and the liftability problem},  Ann. Inst. H. Poincar\'e Anal. Non Lin\'eaire  {\bf 30}
(2013), 101-120.

\bibitem{ADLP16} J. F. Alves, C. L. Dias, S. Luzzatto, V. Pinheiro, {\em SRB measures for partially hyperbolic systems whose
central direction is weakly expanding}, J. Eur. Math. Soc. (JEMS) {\bf 19}
(2016), 2911-2946.





\bibitem{AP10} J. F. Alves, V. Pinheiro, {\em Gibbs-Markov structures and limit laws for partially hyperbolic
 attractors with mostly expanding central direction}, Adv. Math. {\bf 223} (2010), 1706-1730.


\bibitem{AV13} J. F. Alves, H. Vilarinho, {\em Strong stochastic
stability for non-uniformly expanding maps},
Ergod. Th. \& Dynam. Sys. {\bf 33} (2013), 647--692.
%



\bibitem{A98} L.~Arnold.
\newblock {\em Random dynamical systems}.
\newblock Springer Monographs in Mathematics. Springer-Verlag, Berlin (1998).


\bibitem{BaBeM02} V. Baladi, M.
Benedicks, V. Maume-Deschamps, {\em Almost sure rates of mixing
for i.i.d. unimodal maps}, Ann. Scient. \'Ec. Norm. Sup., $4^e$
s\'erie, {\bf 35} (2002), 77--126.





\bibitem{BaV96} V. Baladi, M. Viana, {\em Strong
stochastic stability and rate of mixing for unimodal maps}, Ann.
Scient. \'Ec. Norm. Sup., $4^e$ s\'erie, {\bf 29} (1996), 483--517.


\bibitem{BBD14} W. Bahsoun, C. Bose and Y. Duan, Decay of correlation for random
intermittent maps, {\em Nonlinearity} \textbf{27} (2014) 1543--1554.




\bibitem{BC85} M. Benedicks and L. Carleson, {\em On iterations of
$1-ax^2$
on (-1, 1)}, Ann. Math.{\bf 122} (1985), 1--25.

\bibitem{BC91} \bysame, {\em The
dynamics of the H\'enon map}, Ann. Math. {\bf 133} (1991), 73--169.






%
%





%








\bibitem{CE80} P. Collet, J. Eckmann, {\em On the abundance of
aperiodic behavior for maps on the interval}, Comm. Math. Phys. {\bf73} (1980),
115--160.

\bibitem{CoY05} W. Cowieson, L.-S. Young. {\em  SRB measures as
zero-noise limits}. Ergodic Theory Dynam. Systems  {\bf25} (2005), no. 4, 1091--1113.






\bibitem{K90a}
G.~Keller.
\newblock \emph{Exponents, attractors and {H}opf decompositions for interval maps}.
\newblock   Ergodic Theory Dynam. Systems , \textbf{10}(1990), no 4, 717--744.




 \bibitem{Ki86} Yu. Kifer, {\em  Ergodic theory of
 random perturbations}, Birkh\"auser, Boston Basel, 1986.





 \bibitem{KKM19} A. Korepanov, Z. Kosloff, I. Melbourne, \emph{Explicit coupling argument for nonuniformly hyperbolic transformations}.
Proc. Roy. Soc. Edinburgh \textbf{149} (2019) 101-130.

\bibitem{LV18}  X. Li, H. Vilarinho, \emph{ Almost sure mixing rates for non-uniformly expanding maps}.
Stoch. Dyn., \textbf{18}(2018), no. 4.



\bibitem{LQ95} P.-D. Liu, M. Qian, {\em Smooth Ergodic Theory of Random
Dynamical Systems}, Springer Verlag, Heidelberg, (1995).

\bibitem{MN05} I. Melbourne, M. Nicol, \emph{Almost sure invariance principle for nonuniformly hyperbolic systems}. Comm. Math. Phys. \textbf{260} (2005), no. 1, 131--146.

\bibitem{MN08}  \bysame, \emph{Large deviations for nonuniformly hyperbolic systems}. Trans. Amer. Math. Soc. \textbf{360} (2008), no. 12, 6661--6676.

\bibitem{MN09}  \bysame, \emph{A vector-valued almost sure invariance principle for hyperbolic dynamical systems}. Ann. Probab. \textbf{37} (2009), no. 2, 478--505.


\bibitem{O83} T. Ohno, {\em  Asymptotic behaviors of dynamical systems with random parameters}, Publ. RIMS Kyoto Univ.
{\bf 19} (1983), 83--98.

\bibitem{O68}
V.~I. Oseledec.
\newblock {\em A multiplicative ergodic theorem. {C}haracteristic {L}japunov,
  exponents of dynamical systems}.
\newblock Trudy Moskov. Mat. Ob\v s\v c., \textbf{19} (1968), 179--210.




\bibitem{T01}
H.~Thunberg.
\newblock Unfolding of chaotic unimodal maps and the parameter dependence of
  natural measures.
\newblock {\em Nonlinearity}, \textbf{14}(2), (2001), 323--337.


\bibitem{Y98}  L.-S. Young, {\em Statistical properties of dynamical
systems with some hyperbolicity}, Ann.  Math.  {\bf 147} (1998),
585--650.

\bibitem{Y99} \bysame, {\em Recurrence times and rates of mixing},
Israel J. Math {\bf 110} (1999), 153--188.

\end{thebibliography}

\end{document}